\documentclass[12pt,a4paper]{article}

\usepackage{authblk}
\usepackage[margin=2cm]{geometry}
\usepackage{t1enc}
\usepackage[utf8]{inputenc}
\usepackage{amsthm,amsmath,amssymb}
\usepackage{graphicx}
\usepackage{enumerate}
\usepackage{hyperref}
\usepackage{bm}
\usepackage{comment}
\usepackage{amsfonts}
\usepackage{graphicx,caption}
\usepackage{bm}
\usepackage{amsmath, amsthm, amssymb}
\usepackage{graphicx}
\usepackage{hyperref}
\usepackage{relsize}
\usepackage{blkarray}
\usepackage{algpseudocode}
\usepackage{mathtools}

\usepackage{bbm}

\theoremstyle{plain}
\usepackage{amsthm}
\makeatletter
\newcommand{\newreptheorem}[2]{\newtheorem*{rep@#1}{\rep@title}\newenvironment{rep#1}[1]{\def\rep@title{#2 \ref*{##1}}\begin{rep@#1}}{\end{rep@#1}}}
\makeatother

\newtheorem{theorem}{Theorem}[section]
\newtheorem*{theorem-non}{Theorem}
\newtheorem*{non-lemma}{Lemma}
\newtheorem{lemma}[theorem]{Lemma}
\newreptheorem{lemma}{Lemma}

\newtheorem{claim}[theorem]{Claim}

\theoremstyle{definition}

\newtheorem{definition}[theorem]{Definition}

\DeclareMathOperator{\Pois}{Pois}
\DeclareMathOperator{\Ver}{Ver}

\DeclareMathOperator{\Bd}{Binom}

\DeclareMathOperator{\GW}{GW}

\DeclareMathOperator{\RowS}{RowSpace}

\DeclareMathOperator{\rank}{rank}

\begin{document}
\title{The persistent homology of the Linial-Meshulam process}
\author{Andr\'as M\'esz\'aros}
\date{}
\affil{HUN-REN Alfr\'ed R\'enyi Institute of Mathematics, Budapest, Hungary}
\maketitle
\begin{abstract} 
For a fixed dimension $k\ge 1$, let us consider the randomly growing simplical complex on the vertex set $\{1,2,\dots,n\}$ defined as follows: We start with the empty complex, and for each $k+1$-element subset $\sigma$ of $\{1,2,\dots,n\}$, we add $\sigma$ and all of its subsets to the complex at some random time $t_\sigma$, where $(t_\sigma)$ are i.i.d. uniform random elements of $[0,n]$. As the complex evolves, new $k-1$-dimensional cycles are born and then at a later time they die, that is, they get filled in. The notion of persistence diagrams, which is a standard tool in topological data analysis, provides a way to record these birth and death times. In this paper, we understand the asymptotic behavior of the persistence diagrams of the above defined randomly evolving complexes as $n$ goes to infinity.

As the single time marginals of the above process are variants of the Linial-Meshulam complex, our results can be viewed as extensions of the results of Linial and Peled on the Betti numbers of the Linial-Meshulam complex.

Our proof relies on the notion of local weak convergence of graphs and a generalization of the results of Bordenave,  Lelarge and Salez on the rank of sparse random matrices.

\end{abstract}

\section{Introduction}

Let $\mathcal{K}$ be a simplicial complex, and let $\kappa:\mathcal{K}\to[0,\infty)$ be a function which is increasing in the sense that $\kappa(\tau)\le \kappa(\sigma)$ for all $\tau\subset \sigma\in \mathcal{K}$. 

For $t\ge 0$, we define the subcomplex $\mathcal{K}(t)=\mathcal{K}(\kappa,t)$ of $\mathcal{K}$ as
\[\mathcal{K}(t)=\{\sigma\in \mathcal{K}\,:\,\kappa(\sigma)\le t\}.\]
We refer to $\left(\mathcal{K}(t)\right)_{t\ge 0}$ as the \emph{$\kappa$-filtration} of $\mathcal{K}$. Clearly, $\mathcal{K}(t_1)\subset \mathcal{K}(t_2)\subset \mathcal{K}$ whenever $0\le t_1\le t_2$.

Next, we define the \emph{persistent Betti-numbers} of filtrations of simplicial complexes. In this paper, we only consider reduced homology with real coefficients. For any $t\ge 0$, we have $\mathcal{K}(t)\subset \mathcal{K}$. Thus, $C_i(\mathcal{K}(t)),B_i(\mathcal{K}(t))$ and $Z_i(\mathcal{K}(t))$ can all be considered as a subspace of $C_i(\mathcal{K})$. Relying on the above identification, for any $r,s\ge 0$, we can define the persistent Betti-number 
\[\beta^{r,s}_i(\mathcal{K},\kappa)=\dim \frac{Z_i(\mathcal{K}(r))}{Z_i(\mathcal{K}(r))\cap B_i(\mathcal{K}(s))}.\]

Although we are usually interested in the case $r\le s$, the definition makes sense even for $r>s$. For $r=s$, the persistent Betti number $\beta^{r,s}_i(\mathcal{K},\kappa)
$ just gives back the usual $i$th Betti number of $\mathcal{K}(s)$.

Now, we introduce our main object of study, the \emph{Linial-Meshulam filtration}.

Fix a dimension $k\ge 1$. For a positive integer $n\ge k+1$, let $\mathcal{Y}_n$ be the complete $k$-dimensional simplicial complex on the vertex set $[n]=\{1,2,\dots,n\}$. Let $\kappa_n:\mathcal{Y}_n\to [0,\infty)$ be a random function such that
\begin{enumerate}[\indent (a)]
\item $\left\{\kappa_n(\sigma)\,:\,\sigma\in {{[n]}\choose{k+1}}\right\}$ are i.i.d. uniform random elements of the interval $[0,n]$.
\item For a $\tau\in \mathcal{Y}_n$ such that $|\tau|<k+1$, we have
\[\kappa_n(\tau)=\min_{\tau\subset \sigma \in {{[n]}\choose{k+1}}} \kappa_n(\sigma).\]
\end{enumerate}

Note that for a fixed time $0\le t\le n$, the random complex $\mathcal{Y}_n(t)$ can be also described as follows. Independently for each $k+1$-element subset $\sigma$ of $[n]$, with probability $\frac{t}n$ add all the subsets of $\sigma$ to the complex. Thus, $\mathcal{Y}_n(t)$ is a slightly modified version of the usual well-studied \emph{Linial-Meshulam complex}~\cite{linial2006homological,lm1,lm2,lm3,lm4,lm45,lm5,lm6,lm7}, which is the high dimensional analogue of the Erd\H{o}s-R\'enyi random graph. Indeed, the Linial-Meshulam complex is usually defined to have a complete $k-1$-skeleton, but in our version a lower dimensional face is only added to the complex if we have chosen a $k$-dimensional face containing it. (We would not want to consider the version with a complete $k-1$-skeleton, because in that case all the $k-1$ dimensional cycles would be born at time $0$, thus, the persistent homology would be less interesting.)

Note that since $\kappa_n$ is random, the persistent Betti number $\beta_{k-1}^{r,s}(\mathcal{Y}_n,\kappa_n)$ is a random variable. Our first theorem describes the limiting behavior of these random variables. This result can be viewed as an extension of the results of Linial and Peled~\cite{linial2016phase} on the Betti numbers of the Linial-Meshulam complexes, see Section~\ref{secLP} for more details.

For $q\in (0,1]$ and $c\in [0,\infty)$, let
\begin{align*}\Lambda_{q,c}(t)&=\exp(-c(1-qt)^k)-\frac{c}{q(k+1)}\left(1-(1-qt)^{k+1}-q(k+1)t(1-qt)^k\right),\\
\lambda_{q,c}&=\max_{t\in[0,1]}\Lambda_{q,c}(t).
\end{align*}

\begin{theorem}\label{thm1}
    Let $r,s\in [0,\infty)$. As $n$ tends to infinity, the random variables
    \[\frac{\beta_{k-1}^{r,s}(\mathcal{Y}_n,\kappa_n)}{{{n}\choose{k}}}\]
    converge in probability to some constant $\hat{\beta}_{k-1}^{r,s}$ defined as
    \[\hat{\beta}_{k-1}^{r,s}=\begin{cases}
        \lambda_{1,s}-\exp(-r)\lambda_{\exp(-r),s-r}&\text{if }r< s,\\
        \lambda_{1,s}-\exp(-r)&\text{if }r\ge s.
    \end{cases}\]
    
\end{theorem}

Let $\mathcal{K}$ be a simplical complex and $\kappa:\mathcal{K}\to[0,\infty)$ be an increasing function. Assume that $H_i(\mathcal{K})$ is trivial. The \emph{verbose persistence diagram} $\xi_{\Ver,i}=\xi_{\Ver,i}(\mathcal{K},\kappa)$ of the $\kappa$-filtration of $\mathcal{K}$ is defined as the unique probability measure on 
\[\overline{\Delta}=\{(x,y)\in [0,\infty)^{2}\,:\,x\le y\}\]
such that for all $r,s\in [0,\infty)$, we have
\[\xi_{\Ver,i}(([0,r]\times [0,s])\cap\overline{\Delta})=\frac{\dim Z_i(\mathcal{K}(r))\cap B_i(\mathcal{K}(s))}{\dim Z_i(\mathcal{K})}=\frac{\dim Z_i(\mathcal{K}(r))-\beta_i^{r,s}(\mathcal{K},\kappa)}{\dim Z_i(\mathcal{K})}.\]
The \emph{persistence diagram} $\xi_{i}(\mathcal{K},\kappa)$ of the $\kappa$-filtration of $\mathcal{K}$ is defined as the restriction of $\xi_{\Ver,i}(\mathcal{K},\kappa)$ to 
\[{\Delta}=\{(x,y)\in [0,\infty)^{2}\,:\,x< y\}.\]
See Section~\ref{secPD} for more details.

Let $\hat{\xi}_{\Ver,k-1}$ be the unique probability measure on $\overline{\Delta}$ such that for all $r,s\in [0,\infty)$, we have
\[\hat{\xi}_{\Ver,k-1}(([0,r]\times [0,s])\cap\overline{\Delta})=1-\exp(-r)-\hat{\beta}^{r,s}_{k-1},\]
and let $\hat{\xi}_{k-1}$ be the restriction of $\hat{\xi}_{\Ver,k-1}$ to $\Delta$.

\begin{theorem}\label{diagramconv}\hfill
\begin{enumerate}[(a)]
    \item As $n$ tends to infinity, $\xi_{\Ver,k-1}(\mathcal{Y}_n,\kappa_n)$ converges weakly to $\hat{\xi}_{\Ver,k-1}$ in probability. 

    In other words, let $\mathcal{U}$ be an open set of probability measures on $\overline{\Delta}$ with respect to the weak topology such that $\hat{\xi}_{\Ver,k-1}\in \mathcal{U}$. Then $\lim_{n\to\infty}\mathbb{P}(\xi_{\Ver,k-1}(\mathcal{Y}_n,\kappa_n)\in \mathcal{U})=1$.
    \item As $n$ tends to infinity, $\xi_{k-1}(\mathcal{Y}_n,\kappa_n)$ converges weakly to $\hat{\xi}_{k-1}$ in probability. 
\end{enumerate}
\end{theorem}

Note that part $(a)$ follows from Theorem~\ref{thm1} by a standard argument. Then it also follows clearly  from part $(a)$ that $\xi_{k-1}(\mathcal{Y}_n,\kappa_n)$ converges vaguely to $\hat{\xi}_{k-1}$ in probability. However, to upgrade the vague convergence to weak convergence,  some additional work is required, as we need to exclude the possibility that in the limit some mass escapes through the diagonal~$\overline{\Delta}\setminus \Delta$. 

As a by-product of the proof, we obtain the formula
\begin{equation}\label{totalDelta}\hat{\xi}_{k-1}(\Delta)=1-\frac{1}{k+1}\sum_{j=1}^{k+1}\frac{1}{j}.\end{equation}

In terms of observables, part (a) of Theorem~\ref{diagramconv} says that for any continuous bounded function $f:\overline{\Delta}\to \mathbb{R}$, the integral $\int f d \xi_{\Ver,k-1}(\mathcal{Y}_n,\kappa_n)$ converges to $\int f d \hat{\xi}_{\Ver,k-1}$ in probability. Often we are interested in the integrals of functions which are not bounded. For example, integrating the function $f(r,s)=s-r$ against the measure $\xi_{\Ver,k-1}(\mathcal{Y}_n,\kappa_n)$ gives the normalized lifetime sum, which quantity plays a key role in the higher dimensional generalizations of Frieze's $\zeta(3)$-limit theorem \cite{frieze1985value,hiraoka2017minimum,hino2019asymptotic}.  Our next theorem allows us to determine the limits of integrals of continuous functions $f(r,s)$ where $|f(r,s)|$ is bounded by a polynomial of $s$.

\begin{theorem}\label{thmunboundedobs}
Let $f:\overline{\Delta}\to \mathbb{R}$ be a continuous function. Assume that there is a $0<\ell<\infty$ such that
\[\sup_{(r,s)\in \overline{\Delta}} \frac{|f(r,s)|}{1+s^\ell}<\infty.\]
Then $\int f d \xi_{\Ver,k-1}(\mathcal{Y}_n,\kappa_n)$ converges to $\int f d \hat{\xi}_{\Ver,k-1}$ in probability.

The same result holds with $\xi_{k-1}(\mathcal{Y}_n,\kappa_n)$, $\hat{\xi}_{k-1}$, ${\Delta}$ in place of $\xi_{\Ver,k-1}(\mathcal{Y}_n,\kappa_n)$, $\hat{\xi}_{\Ver,k-1}$, $\overline{\Delta}$, respectively.

\end{theorem}

\subsection{Persistence diagrams}\label{secPD}

In this section, we provide a short and self-contained construction of (verbose) persistence diagrams tailored to our needs. For different or more general approaches, see \cite{edelsbrunner2008persistent,weinberger2011persistent,fugacci2016persistent,adler2010persistent} and the references therein, and also the last paragraph of this section.

Let $\left(\mathcal{K}(t)\right)_{t\ge 0}$ be a $\kappa$-filtration of a simplicial complex $\mathcal{K}$.  We think of the parameter $t$ as time. We assume that $H_i(\mathcal{K})$ is trivial.

For any cycle $ c\in Z_i(\mathcal{K})$, the \emph{birth time} $b(c)$ of $c$ is defined as the time when the hole $c$ first appears, and the \emph{death time} $d(c)$ of $c$ is defined as the time when the hole $c$ is filled in. 

More formally,
\begin{align*}
    b(c)&=\min \{t\ge 0\,:\, c\in Z_i(\mathcal{K}(t))\},\\
    d(c)&=\min \{t\ge 0\,:\,c\in B_i(\mathcal{K}(t))\}.
\end{align*}

These definitions rely on the identification of $Z_i(\mathcal{K}(t))$ and $B_i(\mathcal{K}(t))$ with the corresponding subspaces of ${C}_i(\mathcal{K})$ as discussed in the introduction. Note that $d(c)$ is finite, since we assume that $H_i(\mathcal{K})$ is trivial.

Given any $r,s\in [0,\infty]$, the set of cycles that are born no later than $r$ and die no later than~$s$ form a subspace of $Z_i(\mathcal{K})$, namely,
\[\{c\in Z_i(\mathcal{K})\,:\, b(c)\le r\text{ and }d(c)\le s\}=Z_i(\mathcal{K}(r))\cap B_i(\mathcal{K}(s)).\]
The next lemma shows that we can always find a basis of $Z_{i}(\mathcal{K})$ that behaves nicely with respect to the subspaces $Z_i(\mathcal{K}(r))\cap B_i(\mathcal{K}(s))$. We give the proof of this lemma in Section~\ref{seclemmagoodbasis}.

\begin{lemma}\label{lemmagoodbasis}
    There is a basis $\mathcal{C}$ of $Z_i(\mathcal{K})$ such that for all $r,s\in[0,\infty]$, $\mathcal{C}\cap Z_i(\mathcal{K}(r))\cap B_i(\mathcal{K}(s))$ forms a basis of $Z_i(\mathcal{K}(r))\cap B_i(\mathcal{K}(s))$. In particular, \begin{equation}\label{eqgoodbases}\dim Z_i(\mathcal{K}(r))\cap B_i(\mathcal{K}(s))=|\mathcal{C}\cap Z_i(\mathcal{K}(r))\cap B_i(\mathcal{K}(s))|.\end{equation}
\end{lemma}

The verbose persistent diagram $\xi_{\Ver,i}=\xi_{\Ver,i}(\mathcal{K},\kappa)$ is defined as the probability measure
\[\xi_{\Ver,i}=\frac{1}{\dim Z_i(\mathcal{K})} \sum_{c\in \mathcal{C}} \delta_{(b(c),d(c))},\]
where $\delta_{(b(c),d(c))}$ is the Dirac-measure on $(b(c),d(c))$. Since $0\le b(c)\le d(c)$, the support of the measure $\xi_{\Ver,i}$ is contained in 
\[\overline{\Delta}=\{(x,y)\in [0,\infty)^{2}\,:\,x\le y\}.\]

By equation \eqref{eqgoodbases}, for all $r,s\in [0,\infty]$, we have
\begin{equation}\label{measurevsBetti0}\xi_{\Ver,i}(([0,r]\times [0,s])\cap\overline{\Delta})=\frac{\dim Z_i(\mathcal{K}(r))\cap B_i(\mathcal{K}(s))}{\dim Z_i(\mathcal{K})}.\end{equation}
It is a standard fact that the equations in \eqref{measurevsBetti0} uniquely determine the measure $\xi_{\Ver,i}$. Note that the right hand side of \eqref{measurevsBetti0} does not depend on $\mathcal{C}$. Thus, $\xi_{\Ver,i}$ does not depend on the choice of the basis $\mathcal{C}$. As before, persistence diagrams are simply defined as the restrictions of verbose persistence diagrams to $\Delta$.

The persistent Betti numbers defined in the introduction can be expressed with the help of the verbose persistence diagram as follows:
\begin{align*}\frac{\beta_i^{r,s}(\mathcal{K},\kappa)}{\dim Z_i(\mathcal{K})}&=\frac{\dim Z_i(\mathcal{K}(r))-\dim Z_i(\mathcal{K}(r))\cap B_i(\mathcal{K}(s))}{\dim Z_i(\mathcal{K})}
\\&=\xi_{\Ver,i}(([0,r]\times [0,\infty))\cap\overline{\Delta})-\xi_{\Ver,i}(([0,r]\times [0,s])\cap\overline{\Delta})\\&=\xi_{\Ver,i}(([0,r]\times (s,\infty))\cap\overline{\Delta}).
\end{align*}
This formula is true for all choices of $r,s\in [0,\infty)$. Provided that $r\le s$, the formula remains true even if we replace the verbose persistence diagram with the persistence diagram. 
 This formula also shows that our definition of verbose persistence diagrams is equivalent to other possible definitions, see for example \cite[Proposition 4.8.]{joe2025limit}. 

 In our setup, the notion of verbose persistence diagrams is more natural than the notion of persistence diagrams. However, there are other ways to build the theory of persistent homology, where this is the other way around. For example, persistence diagrams can be defined as follows. For $0\le r\le s$, the natural inclusion $\iota^{r,s}:\mathcal{K}(r)\xhookrightarrow{} \mathcal{K}(s)$ induces a homomorphism $\iota_*^{r,s}:H_i(\mathcal{K}(r))\to H_i(\mathcal{K}(s))$. The homology groups $\left(H_i(\mathcal{K}(r))\right)_{0\le r}$ together with the homomorphisms  $\left(\iota_*^{r,s}\right)_{0\le r\le s}$ form an algebraic structure called the \emph{persistence module}. This persistence module can be decomposed as the direct sum of interval modules~\cite{crawley2015decomposition}. Each summand in this decomposition corresponds to an atom of the persistence diagram. In particular, the persistence diagram can be obtained from the persistence module. However, the atoms of the verbose  persistence diagram lying on the diagonal cannot be recovered from only the persistence module.

\subsection{Topological data analysis -- A motivation for studying persistent homology}

Let $S$ be a finite set of points in $\mathbb{R}^m$. We would like to make a formal sense of the question: What is the shape of the set $S$? One possibility is to choose a radius $r\ge 0$ and consider the union of the $r$-balls centered around the points of $S$, that is, to consider the set
\[S_r=\bigcup_{s\in S} B_r(s),\]
then try to understand the homological properties of this set $S_r$. Of course, we should choose the radius $r$ carefully. If $r$ is too small, then $S_r$ is just a disjoint union of balls. If $r$ is too large, then $S_r$ is contractible. To overcome the problem of the choice of the radius $r$, we can just consider all the possible choices of $r$, and try to understand the evolution of $S_r$ as $r$ grows. As we increase $r$, new holes can appear in $S_r$, while other holes might get filled in. We can use persistent homology to capture the changes of the homological features of $S_r$. 

\emph{\v{C}ech complexes} are useful to understand the homology of $S_r$. Given $S$ and $r$, the \v{C}ech complex $\mathcal{K}(r)$ is defined as the simplicial complex on the vertex set $S$, where a subset $\sigma$ of $S$ is present in the complex~$\mathcal{K}(r)$ if and only if the $r$-balls centered around the points of $\sigma$ have a nonempty intersection. By the nerve lemma, the \v{C}ech complex $\mathcal{K}(r)$ is homotopy equivalent to~$S_r$. Thus, the homology of $\mathcal{K}(r)$ is the same as the homology of $S_r$. Note that the complexes~$\mathcal{K}(r)$ can be obtained as a filtration of the simplex on the vertex set $S$ by setting
\[\kappa(\sigma)=\min_{o\in \mathbb{R}^m} \max_{s\in \sigma} \|o-s\|_2.\]

\emph{Topological data analysis} studies the persistent homology of point sets coming from real life data. An important question is whether a certain feature of the persistent homology tells us something significant  about our dataset or  is just a result of some noise in the data. Thus, we would like to be able to compare the behavior of our dataset with a random one. This motivates the study of the persistent homology of random filtrations.

The model studied by us along with \emph{clique complexes}~\cite{ababneh2024maximal,kahle2009topology,kahle2014sharp,malen2023collapsibility} can be considered as the simplest mean-field model of a random filtration. Filtrations defined from random point processes were also studied a lot~\cite{bobrowski2022homological,bobrowski2017maximally,bobrowski2023universal,edelsbrunner2024maximum,hiraoka2018limit}. Although we have strong results on the limiting behavior of the persistent homology of these geometric complexes, we do not have such explicit formulas in these cases as in Theorem~\ref{thm1}. 

\subsection{Recovering the results of Linial and Peled}\label{secLP}

Let $\overline{\mathcal{Y}_n}(c)$ be obtained from  ${\mathcal{Y}_n}(c)$ by adding all the subsets of $[n]$ of size at most $k$. That is, $\overline{\mathcal{Y}_n}(c)$ is just the usual Linial-Meshulam complex, which was the model studied by Linial and Peled~\cite{linial2016phase}.

Let $c>0$. A straightforward calculation shows that $\Lambda_{1,c}'(t)=0$ if and only if $t$ satisfies the fixed point equation $t=\exp(-c(1-t)^k)$. Moreover, $\Lambda_{1,c}'(0)>0$ and $\Lambda_{1,c}'(1)=0$. Thus, the maximum of $\Lambda_{1,c}$ in the interval $[0,1]$ is attained at one of the elements of the set \[F_c=\left\{t\in [0,1]\,:\,t=\exp(-c(1-t)^k)\right\}.\]

Therefore,
\begin{align}
\lambda_{1,c}&=\max_{t\in F_c}\left(\exp(-c(1-t)^k)-\frac{c}{(k+1)}\left(1-(1-t)^{k+1}-(k+1)t(1-t)^k\right)\right)\nonumber\\&=\max_{t\in F_c}\left(t-\frac{c}{(k+1)}\left(1-(1-t)^{k+1}-(k+1)t(1-t)^k\right)\right).\label{lpe0}
\end{align}

Also, observe that
\begin{multline}\label{lpe1}\dim H_k(\overline{\mathcal{Y}_n}(c))=\dim Z_k(\overline{\mathcal{Y}_n}(c))=\dim Z_k({\mathcal{Y}_n}(c))=\dim C_{k}({\mathcal{Y}_n}(c))-\dim B_{k-1} ({\mathcal{Y}_n}(c))\\=\dim C_{k}({\mathcal{Y}_n}(c))-\dim Z_{k-1}({\mathcal{Y}_n}(c))+\dim H_{k-1}(\mathcal{Y}_n(c)).\end{multline}

With some effort, one can see that
\begin{equation}\label{lpe2}\lim_{n\to\infty}\frac{\dim C_{k}({\mathcal{Y}_n}(c))-\dim Z_{k-1}({\mathcal{Y}_n}(c))}{{{n}\choose{k}}}=\frac{c}{k+1}-1+\exp(-c)\quad\text{ in probability,}\end{equation}
see also~\eqref{Zlimit}.

Combining \eqref{lpe1} and \eqref{lpe2} with Theorem~\ref{thm1}, we see that
\[\lim_{n\to\infty}\frac{\dim H_k(\overline{\mathcal{Y}_n}(c))}{{{n}\choose{k}}}=\frac{c}{k+1}-1+\lambda_{1,c} \quad\text{ in probability.}\]

Here, by \eqref{lpe0}, the right hand side can be written as
\begin{equation}\label{LiPemax}\max_{t\in F_c}\left(ct(1-t)^k+\frac{c}{(k+1)}(1-t)^{k+1}-(1-t)\right).\end{equation}

Linial and Peled~\cite{linial2016phase} proved that there is a $c_k^*$ such that if $c\le c_k^*$, then the maximum in \eqref{LiPemax} is attained at $t=1$, and if $c>c_k^*$, then the maximum is attained at the smallest element of $F_c$. Thus, we have the following theorem:
\begin{theorem}[Linial and Peled~\cite{linial2016phase}]\hfill

\begin{enumerate}[(a)]
    \item Assume that $c\le c_k^*$, then
    \[ \lim_{n\to\infty}\frac{\dim H_k(\overline{\mathcal{Y}_n}(c))}{{{n}\choose{k}}}=0 \quad\text{ in probability.}\]
    \item Assume that $c> c_k^*$, then
    \[ \lim_{n\to\infty}\frac{\dim H_k(\overline{\mathcal{Y}_n}(c))}{{{n}\choose{k}}}=ct_c(1-t_c)^k+\frac{c}{(k+1)}(1-t_c)^{k+1}-(1-t_c) \quad\text{ in probability,}\]
    where $t_c$ is the smallest solution of $t=\exp(-c(1-t)^k)$.

\end{enumerate}
    
\end{theorem}

\medskip
\textbf{Acknowledgments.} The author was supported by the Marie Skłodowska-Curie Postdoctoral Fellowship "RaCoCoLe".
\section{Preliminaries}

\subsection{Local weak convergence of graphs}\label{seclocalweak}

In this paper, all graphs are assumed to be locally finite, that is, all graphs are assumed to have finite degrees.

A \emph{rooted graph} $(G,o)$ is a (possible infinite, but locally finite) connected graph $G$ together with a distinguished vertex $o$ of~$G$ called the root. The depth of $(G,o)$ is defined as the maximum of the distances of all the vertices from the root. Two rooted graphs are isomorphic, if there is a graph isomorphism between them that also preserves the root. The space consisting of all the isomorphism classes of rooted graphs can be endowed with a measurable structure. Thus, it makes sense to speak about random rooted graphs. We omit the details here, which can be found in \cite{aldous2019processes}. 

Given a graph $G$, a vertex $o$ of $G$ and a radius $r$, we use $B_r(G,o)$ to denote the subgraph of $G$ induced by the vertices of $G$ which are at distance at most $r$ from $o$. Choosing $o$ as the root, we can turn $B_r(G,o)$ into a rooted graph. With a slight abuse of notation, we also use $B_r(G,o)$ to denote this rooted graph.

We say that a sequence of random rooted graphs $(G_n,o_n)$ converges locally to a random rooted graph $(G,o)$ if for all non-negative integers $r$, and all rooted graphs $(H,q)$ of depth at most $r$, we have
\[\lim_{n\to\infty}\mathbb{P}(B_r(G_n,o_n)\cong (H,q))=\mathbb{P}(B_r(G,o)\cong (H,q)).\]
Here, the symbol $\cong$ refers to a rooted isomorphism.

By a set-rooted graph, we mean a pair $(G,U)$, where $G$ is a finite graph and $U$ is a non-empty subset of the vertices of $G$. With each set-rooted graph $(G,U)$, we can associate a random rooted graph $(G_o,o)$, where $o$ is a uniform random element of $U$, and $G_o$ is the connected component of $o$ in the graph $G$. A sequence of set-rooted graphs $(G_n,U_n)$ converges to a random rooted graph $(G,o)$ if the random rooted graphs associated to $(G_n,U_n)$ converge to $(G,o)$.

Next we give an equivalent definition of the convergence of set-rooted graphs which will be easier to extend to the case when the pair $(G_n,U_n)$ is random.

For a set-rooted graph $(G,U)$, a radius $r$ and rooted graph $(H,q)$ of depth at most $r$, we define
\[t_r(G,U,(H,q))=\frac{|\{o\in U\,:\,B_r(G,o)\cong (H,q)\}|}{|U|}.\]

Note that a sequence of set-rooted graphs $(G_n,U_n)$ converges to a random rooted graph~$(G,o)$ if and only if
\[\lim_{n\to\infty} t_r(G,U,(H,q))=\mathbb{P}(B_r(G,o)\cong (H,q))\]
for all $r$ and $(H,q)$ as above.

Now, let $(G_n,U_n)$ be a sequence of random set-rooted graphs. Then $t_r(G_n,U_n,(H,q))$ is a random variable. We say that $(G_n,U_n)$ converges to a random rooted graph $(G,o)$ in probability if
\[\lim_{n\to\infty} t_r(G,U,(H,q))=\mathbb{P}(B_r(G,o)\cong (H,q))\text{ in probability}\]
for all $r$ and $(H,q)$ as above.

A rooted decorated graph is a triple $((G,S),o)$, where $(G,o)$ is a rooted graph and $S$ is a subset of the vertices of $G$. Two rooted decorated graphs $((G_1,S_1),o_1)$ and $((G_2,S_2),o_2)$ are isomorphic, if there is a graph isomorphism $\phi:V(G_1)\to V(G_2)$ such that $\phi(o_1)=o_2$ and $\phi(S_1)=S_2$. We can extend the notions defined above to decorated graphs as well. First, the space of (isomorphism classes of) rooted decorated graphs can be endowed with a measurable structure~\cite{aldous2019processes}. Thus, we can speak about random rooted decorated graphs.

A set-rooted decorated graph is a triple $((G,S),U)$, where $G$ is a finite graph, $S,U$ are  subsets of the vertices of $G$ such that $U\neq \emptyset$. Given $r$, and a rooted decorated graph  $((H,T),q)$ of depth at most $r$, we define
\[t_r((G,S),U,((H,T),q))=\frac{|\{o\in U\,:\,((B_r(G,o),S\cap V(B_r(G,o))),o)\cong ((H,T),q)\}|}{|U|}.\]

Let $((G_n,S_n),U_n)$ be a sequence of set-rooted decorated graphs, we say that $((G_n,S_n),U_n)$ converges to a random rooted decorated graph $((G,S),o)$ if for all radius $r$, and all rooted decorated graph $((H,T),q)$ of depth at most $r$, we have
\[\lim_{n\to\infty} t_r((G_n,S_n),U,((H,T),q))=\mathbb{P}(((B_r(G,o),S\cap V(B_r(G,o))),o)\cong((H,T),q)).\]

The definition of convergence can be extended to the case when $((G_n,S_n),U_n)$ is random in the same way as in the undecorated case.

The proofs of the lemmas below are straightforward.

\begin{lemma}\label{localweak1}
Let $((G_n,S_n),U_n)$ be a sequence of random set-rooted decorated graphs converging to $((G,S),o)$ in probability.  Let $p\in [0,1]$, and define $S_n'$ as a random subset of $S_n$, where every element of $S_n$ is kept with probability $p$ independently. We define the random subset $S'$ of $S$ analogously. 

Assuming that $|U_n|\to\infty$ in probability, the random set-rooted decorated graphs $((G_n,S_n'),U_n)$ converge to $((G,S'),o)$ in probability.

\end{lemma}

\begin{lemma}\label{localweak2}
Let $((G_n,S_n),U_n)$ be random set-rooted decorated graphs converging to $(
(G,S),o)$ in probability. 
Let
\[S_n'=\cup_{s\in S_n} V(B_1(G_n,s)) \text{ and }S'=\cup_{s\in S} V(B_1(G,s)).\]

Then
$((G_n,S_n'),U_n)$ converges to $((G,S'),o)$ in probability.

\end{lemma}

\begin{lemma}\label{localweak3}
Let $((G_n,S_n),U_n)$ be a sequence of random set-rooted decorated graphs converging  to $((G,S),o)$ in probability. 

Let $G_n'$ be the graph obtained from $G_n$ by deleting all the vertices in $S_n$, and let $U_n'=U_n\setminus S_n$. On the event $o\notin S$, let  $G^-$ be the graph obtained from $G$ by deleting all the vertices in $S$ and only keeping the connected component of the root in the resulting graph. Assume that the event $o\notin S$ has positive probability. Let $(G',o')$ have the distribution of $(G^-,o)$ conditioned on the event that $o\notin S$.

Then $(G_n',U_n')$ converges to $(G',o')$ in probability.
\end{lemma}

\subsection{Multi-type Galton-Watson trees}\label{SecGW}

Given a probability measure $\nu=(\nu_i)_{i\ge 0}$ on the set of nonnegative integers with finite and nonzero mean $\bar{\nu}$, we define the probability measure $\nu'=(\nu_i')_{i\ge 0}$ on the set of nonnegative integers by the formula
\[\nu_i'=\frac{1}{\bar{\nu}}(i+1)\nu_{i+1}.\]
We also consider the generating function
\[f(\nu,t)=\sum_{i=0}^\infty\nu_i t^i.\]

Note that 
\begin{equation*}
f(\nu',t)=\frac{1}{\bar{\nu}}\frac{d}{dt} f(\nu,t).
\end{equation*}

Let $\mu=(\mu_i)_{i\ge 0}$ and $\nu=(\nu_i)_{i\ge 0}$ be two probability measures on the set of nonnegative integers with finite second moments.

We define a $\GW(\mu',\nu')$-tree as a multi-type \emph{Galton-Watson tree} such that if a vertex is at even distance from the root, then the number of children of this vertex is chosen according to~$\mu'$, moreover, if a vertex is at odd distance from the root, then the number of children of this vertex is chosen according to $\nu'$. 

We define a $\GW_*(\mu,\nu)$-tree the same way as a $\GW(\mu',\nu')$-tree except that the number of the children of the root is chosen according to $\mu$ instead of $\mu'$.

\section{The convergence of persistent Betti numbers}

\subsection{A theorem on the rank of sparse of matrices}\label{SecRankThmStatement}

Let $M$ be a matrix over $\mathbb{R}$. We assume that the rows of $M$ are indexed by the set $R$ and the columns of $M$ are indexed by the set $C$, where $R$ and $C$ are disjoint. The \emph{Tanner graph} of $M$ is a bipartite graph $G$ with color classes $R$ and $C$ such that for all $r\in R$ and $c\in C$, $rc$ is an edge of $G$ if and only if $M(r,c)\neq 0$.


Let $\mu=(\mu_i)_{i=0}^\infty$ and $\nu=(\nu_i)_{i=0}^\infty$ be two probability distributions with finite second moments such that $\bar{\mu}>0$ and $\bar{\nu}>0$. We define
\[\Lambda(t)=f(\mu,1-f(\nu',1-t))-\frac{\bar{\mu}}{\bar{\nu}}\left(1-f(\nu,1-t)-\bar{\nu}t f(\nu',1-t)\right),\]
where we relied on the notations of Section~\ref{SecGW}.

Let $M_n$ be a sequence of random matrices over $\mathbb{R}$ such that all the entries of $M_n$ are from the set $\{-1,0,+1\}$. We assume that the rows of $M_n$ are indexed by the set $R_n$ and the columns of $M_n$ are indexed by the set $C_n$, where $R_n$ and $C_n$ are disjoint. Let $G_n$ be the Tanner graph~$M_n$. We say that $M_n$ converges locally to a random rooted graph $(G,o)$ in probability, if the random set-rooted graphs $(G_n,C_n)$ converge to $(G,o)$ in probability, in the sense defined in Section~\ref{seclocalweak}.

\begin{theorem}\label{RankThm}
Assume that $M_n$ converges locally to a $\GW_*(\mu,\nu)$-tree in probability.

Let $\alpha$ and $\alpha'$ be the smallest and largest solution of
 \[t=f(\mu',1-f(\nu',1-t))\]
 in the interval $[0,1]$, respectively. 

 Assume that $\max_{t\in [0,1]}\Lambda(t)=\max(\Lambda(\alpha),\Lambda(\alpha'))$.

Then
\[\lim_{n\to\infty} \frac{\rank(M_n)}{|C_n|}=1-\max_{t\in [0,1]}\Lambda(t) \quad\text{ in probability}.\]

\end{theorem}

We prove this theorem in Section~\ref{SecRankThm}. In the special case of $\mu=\nu$, Theorem~\ref{RankThm} was basically proved by Bordenave, Lelarge and Salez~\cite{bordenave2011rank}. Our proof mainly follows their argument. For random matrices chosen uniformly at random from the set of matrices with prescribed numbers of non-zero entries in each row and column, similar results were proved in \cite{coja2023rank}. We stated Theorem~\ref{RankThm} in a greater generality than what is necessary for our purposes, hoping for further applications.

\subsection{The convergence of persistent Betti numbers}
In this section, we prove Theorem~\ref{thm1}.

Let $J_n$ be the matrix of the $k-1$th coboundary map $\mathcal{Y}_n$. Note that the columns of $J_n$ are indexed by ${[n]}\choose{k}$ and the rows are indexed by ${[n]}\choose{k+1}$. 

Let $K_n$ be the submatrix of $J_n$ determined by the rows indexed by the $k$-dimensional faces of $\mathcal{Y}_n(s)$. Moreover, let $M_n$ be the matrix obtained from $K_n$ by deleting the columns indexed by the $k-1$ dimensional faces of $\mathcal{Y}_n(r)$ and the rows indexed by the $k$-dimensional faces of~$\mathcal{Y}_n(r)$. 

Note that
\begin{multline}\label{dimZcB}\dim Z_{k-1}(\mathcal{K}(r))\cap B_{k-1}(\mathcal{K}(s))\\ =\dim \left\{w\in \RowS(K_n)\,:\, w(\sigma)=0\text{ for all }\sigma\in {{[n]}\choose{k}}\setminus \mathcal{Y}_n(r)\right\}.\end{multline}

Consider the projection $P$ that maps a vector $(w_\sigma)_{\sigma\in {{[n]}\choose{k}}}$ to $(w_\sigma)_{\sigma\in {{[n]}\choose{k}}\setminus \mathcal{Y}_n(r)}$. Let us restrict $P$ to $\RowS(K_n)$. Applying the rank nullity theorem to this map, we obtain that
\begin{multline*}\dim \RowS(K_n)\\=\dim \RowS(M_n)+\dim \left\{w\in \RowS(K_n)\,:\, w(\sigma)=0\text{ for all }\sigma\in {{[n]}\choose{k}}\setminus \mathcal{Y}_n(r)\right\},\end{multline*}
or, in other words,
\[\rank(K_n)=\rank(M_n)+\dim \left\{w\in \RowS(K_n)\,:\, w(\sigma)=0\text{ for all }\sigma\in {{[n]}\choose{k}}\setminus \mathcal{Y}_n(r)\right\}.\]

Combining this with \eqref{dimZcB}, we see that
\begin{equation}\label{dimasrankdiff}
    \dim Z_{k-1}(\mathcal{K}(r))\cap B_{k-1}(\mathcal{K}(s)) =\rank(K_n)-\rank(M_n).
\end{equation}

First, assume that $r<s$. 
Let $c=s-r$, and $q=e^{-r}$.

The next two lemmas are proved in Section~\ref{seclocallimitlemma} and Section~\ref{secfixedpoint}, respectively. 

\begin{lemma}\label{lemmaMnlimit}\hfill
\begin{enumerate}[(a)]
    \item We have \[\lim_{n\to\infty} \frac{\dim C_{k-1}(\mathcal{Y}_n(r))}{{{n}\choose{k}}}=1-q\quad\text{ in probability}.\]
    \item $M_n$ converges locally to a $\GW_*(\Pois(c),\Bd(q,k+1))$-tree. 
\end{enumerate}

\end{lemma}

\begin{lemma}\label{lemmaFixedpoints}
    Given $0<q\le 1$ and $c> 0$, let $\alpha$ and $\alpha'$ be the smallest and largest solution of $t=\exp(-c(1-qt)^k)$ in the interval $[0,1]$. Then
    \[\lambda_{q,c}=\max(\Lambda_{q,c}(\alpha),\Lambda_{q,c}(\alpha')).\]
\end{lemma}

\begin{lemma}\label{asymprank1}
We have \[\lim_{n\to\infty}\frac{\rank(M_n)}{{{n}\choose{k}}}=q(1-\lambda_{q,c})\quad\text{ in probability.}\]
\end{lemma}
\begin{proof}
Recall that the columns of $M_n$ are indexed by $C_n={{[n]}\choose{k}}\setminus \mathcal{Y}_n(r)$. By part (a) of Lemma~\ref{lemmaMnlimit}, we have
\begin{equation}\label{eqtoq}\lim_{n\to\infty} \frac{|C_n|}{{{n}\choose{k}}}=1-\lim_{n\to\infty}\frac{\dim C_{k-1}(\mathcal{Y}_n(r))}{{{n}\choose{k}}}=q\quad\text{ in probability.}\end{equation}

Combining part (b) of Lemma~\ref{lemmaMnlimit} and Lemma~\ref{lemmaFixedpoints} with Theorem~\ref{RankThm}, we see that
\begin{equation}\label{eqar0}\lim_{n\to\infty}\frac{\rank(M_n)}{|C_n|}=1-\lambda_{q,c}\quad \text{ in probability.}\end{equation}

Putting together \eqref{eqtoq} and \eqref{eqar0}, the statement follows.
\end{proof}

Applying the above Lemma~\ref{asymprank1} with $0$ in place of $r$, we obtain the following lemma.

\begin{lemma}\label{asymprank2}
We have \[\lim_{n\to\infty}\frac{\rank(K_n)}{{{n}\choose{k}}}=1-\lambda_{1,s}\quad\text{ in probability.}\]
\end{lemma}

Combining Lemma~\ref{asymprank1} and Lemma~\ref{asymprank2} with \eqref{dimasrankdiff}, we see that for $r<s$, we have
\begin{align*}\lim_{n\to\infty}& \frac{\dim Z_{k-1}(\mathcal{Y}_n(r))\cap B_{k-1}(\mathcal{Y}_n(s))}{{{n}\choose{k}}}\\&=1-\lambda_{1,s}-q(1-\lambda_{q,c})\\&=1-\exp(-r)-(\lambda_{1,s}-\exp(-r)\lambda_{\exp(-r),s-r})\quad\text{ in probability.}
\end{align*}

If $r\ge s$, then $\dim Z_{k-1}(\mathcal{Y}_n(r))\cap B_{k-1}(\mathcal{Y}_n(s))=\dim B_{k-1}(\mathcal{Y}_n(s))$. Thus, by Lemma~\ref{asymprank2}, we have
\[\lim_{n\to\infty} \frac{\dim Z_{k-1}(\mathcal{Y}_n(r))\cap B_{k-1}(\mathcal{Y}_n(s))}{{{n}\choose{k}}}=1-\lambda_{1,s}\quad\text{ in probability.}\]

Comparing these with the definition of $\hat{\beta}^{r,s}_{k-1}$, we see that for any $r,s\in[0,\infty)$, we have
\begin{equation}
    \label{Bettilimit01}
    \lim_{n\to\infty} \frac{\dim Z_{k-1}(\mathcal{Y}_n(r))\cap B_{k-1}(\mathcal{Y}_n(s))}{{{n}\choose{k}}}=1-\exp(-r)-\hat{\beta}_{k-1}^{r,s}\quad\text{ in probability.}
\end{equation}

By the rank-nullity theorem
\[\dim C_{k-1}(\mathcal{Y}_n(r))-\dim Z_{k-1}(\mathcal{Y}_n(r))=\dim B_{k-2}(\mathcal{Y}_n(r))\le {{n}\choose{k-1}}.\]

Thus, it follows from part (a) of Lemma~\ref{lemmaMnlimit} that
\begin{equation}\label{Zlimit}\lim_{n\to\infty} \frac{\dim Z_{k-1}(\mathcal{Y}_n(r))}{{{n}\choose{k}}}=1-\exp(-r)\quad\text{ in probability}.\end{equation}

Combining this with \eqref{Bettilimit01}, we have 
\[\lim_{n\to\infty} \frac{\beta_{k-1}^{r,s}(\mathcal{Y}_n,\kappa_n)}{{{n}\choose{k}}}=\hat{\beta}^{r,s}_{k-1}\quad\text{ in probability.}\]

Therefore, we proved Theorem~\ref{thm1}.





\subsection{The local weak limit of matrices coming from the Linial-Meshulam filtration -- The proof of Lemma~\ref{lemmaMnlimit}}\label{seclocallimitlemma}

Let $(G,o)$ be a $\GW_*(\Pois(s),k+1)$-tree. (Here $k+1$ stands for the Dirac measure on $k+1$.)

Let $G_n$ be the Tanner graph of $K_n$. The two color classes of $G_n$ are $C_n={{[n]}\choose {k}}$ and $R_n=\mathcal{Y}_n(s)\cap {{[n]}\choose {k+1}}$.

By \cite[Theorem 13]{kanazawa2022law}, we see that $(G_n,C_n)$ converges in probability to $(G,o)$. 

Note that
\[\frac{\dim C_{k-1}(\mathcal{Y}_n(s))}{|C_n|}=\frac{|\{c_n\in C_n\,:\,\deg_{G_n}(c_n)>0\}|}{|C_n|}.\]

Thus,
\[\lim_{n\to\infty}\frac{\dim C_{k-1}(\mathcal{Y}_n(s))}{|C_n|}=\mathbb{P}(\deg_G(o)>0)=1-\exp(-s)\text{ in probability.}\]

Replacing $s$ with $r$, part (a) of Lemma~\ref{lemmaMnlimit} follows.

Let $R$ be the set of vertices of $G$, which are at odd distance from $o$. It follows easily that the random set-rooted decorated graphs $((G_n,R_n),C_n)$ converge in probability to $((G,R),o)$. 

Let $p=\frac{r}s$. Let $R_n'$ be the subset of $R_n$, where each element of $R_n$ is kept with probability~$p$, independently. Similarly, let $R'$ be the subset of $R$, where each element of $R$ is kept with probability $p$, independently.

Then, let
\[S_n=\cup_{v\in R_n'} V(B_1(G_n,v))\quad\text{ and }\quad S=\cup_{v\in R'} V(B_1(G,v)).\]

Let $G_n^*$ be the graph obtained from $G_n$ by deleting all the vertices in $S_n$, and let $C_n^*=C_n\setminus S_n$. On the event $o\notin S$, let  $G^-$ be the graph obtained from $G$ by deleting all the vertices in $S$ and only keeping the connected component of the root.  Let $(G',o')$ have the distribution of $(G^-,o)$ conditioned on the event that $o\notin S$.
Combining Lemma~\ref{localweak1}, Lemma~\ref{localweak2} and Lemma~\ref{localweak3}, we see that the random set-rooted graphs $(G_n^*,C_n^*)$ converge to $(G',o')$ in probability.

Let $G_n'$ be the Tanner graph of $M_n$ and let $C_n'={{[n]}\choose 
{k}}\setminus \mathcal{Y}_n(r)$. It is straightforward to see that $(G_n^*,C_n^*)$ has the same distribution as $(G_n',C_n')$. 

Therefore, it is enough to prove that $(G',o')$ is a $\GW_*(\Pois(c),\Bd(q,k+1))$-tree. 

Let $\ell$ be even. Conditioned on $B_\ell(G,o)$ and $R'\cap V(B_\ell(G,o))$, for any vertex $v$ at distance~$\ell$ from $o$, we have that 
\begin{itemize}
\item the number of children of $v$ which are in $R'$ is given by a $\Pois(r)$ random variable $A_v$;
\item  the number of children of $v$ which are not in $R'$ is given by a $\Pois(s-r)$ random variable~$B_v$;
\item The random variables $A_v,B_v$ are independent for all choices of vertices $v$ at distance $\ell$ from $o$.
\end{itemize}

Therefore, $\mathbb{P}(A_v=0)=\exp(-r)=q$. Moreover, conditioned on the event that $A_v=0$, the law of $B_v$ is $\Pois(s-r)$. Let $w$ be the parent of $v$. It follows that $\mathbb{P}(v\notin S\,|\, w\notin R')=q$, and conditioned on the $v\notin S$, the number of children of $v$ is given by a $\Pois(s-r)$ random variable.

A similar argument gives that conditioned on $B_{\ell-1}(G,o)$ and $R'\cap V(B_{\ell-1}(G,o))$, if $u$ is at distance $\ell-1$ from $o$ such that $u\notin R'$, the number of children of $u$ which are not in $S$ is given by a $\Bd(k,q)$ random variable. Moreover, these random variables are independent for all choices of $u$.

Part (b) of Lemma~\ref{lemmaMnlimit} follows from these observations.

\subsection{The fixed points of the function $\exp(-c(1-qt)^k)$ -- The proof of Lemma~\ref{lemmaFixedpoints}}\label{secfixedpoint}

First, let us fix $q\in (0,1)$. Note that
\begin{equation}\label{Lambdadiff}\Lambda_{q,c}'(t)=ckq(1-qt)^{k-1}(\exp(-c(1-qt)^k)-t).\end{equation}
Thus,  $\Lambda_{q,c}'(t)$ has the same sign as $\phi_{c}(t)=\exp(-c(1-qt)^k)-t$. 



For $t\in (0,1]$, let \[\Psi(t)=\frac{-\log(t)}{(1-qt)^k},\]
we also set $\Psi(0)=+\infty=\lim_{t\to 0_+} \Psi(t)$.

It is straightforward to see that for $t\in [0,1]$, we have
\begin{enumerate}[\indent (a)]
\item We have $\phi_c(t)=0$, if and only if $c=\Psi(t)$.
\item We have $\phi_c(t)<0$, if and only if $c>\Psi(t)$.
\item We have $\phi_c(t)>0$, if and only if $c<\Psi(t)$.
\end{enumerate}

We have
\[\Psi'(t)=\frac{qt-1-kqt\log(t)}{t(1-qt)^{k+1}}.\]

Note that
\[\frac{d}{dt}\left(qt-1-kqt\log(t)\right)=q-kq-kq\log(t)\]
 is a monotonically decreasing function. 
Thus, $\Psi'(t)$ can have at most two zeros in $[0,1]$. Also note that $\Psi$ is nonnegative, $\Psi(0)=\infty$ and $\Psi(1)=0$. It follows from these observations that either
\begin{enumerate}[\indent (a)]
    \item The function $\Psi(t)$ monotonically decreasing on $[0,1]$; or
    \item We have $0<s_1<s_2<1$ such that  $\Psi(t)$ monotonically decreasing on $[0,s_1]$, increasing on $[s_1,s_2]$, decreasing on $[s_2,1]$.
\end{enumerate}
In case (a), for each $c$, $\phi_c(t)=0$ has a unique solution $t_c$ in $[0,1]$, and $\lambda_{q,c}=\Lambda_{q,c}(t_c)$.

Case (b) is slightly more complicated:
\begin{itemize}
    \item If $c>\Psi(s_2)$, then $\phi_c(t)=0$ has a unique solution $t_c$ in $[0,1]$, and $\lambda_{q,c}=\Lambda_{q,c}(t_c)$.
    \item If $c=\Psi(s_2)$, then $\phi_c(t)=0$ has two solutions $0<t_c<t_c'<1$ in $[0,1]$. Moreover, $t_c'=s_2$, and $\lambda_{q,c}=\Lambda_{q,c}(t_c)$.

    \item If $\Psi(s_1)<c<\Psi(s_2)$, then $\phi_c(t)=0$ has three solutions $0<t_c<t_c'<t_c''<1$ in~$[0,1]$. Moreover, $\Lambda_{q,c}$ has local maximums at $t_c$ and $t_c''$, and a local minimum at $t_c'$. Thus,
    $\lambda_{q,c}=\max (\Lambda_{q,c}(t_c),\Lambda_{q,c}(t_c''))$.
    \item If $c=\Psi(s_1)$, then $\phi_c(t)=0$ has two solutions $0<t_c<t_c'<1$ in $[0,1]$. Moreover, $t_c=s_1$, and $\lambda_{q,c}=\Lambda_{q,c}(t_c')$.
    
    \item If $c<\Psi(s_1)$, then $\phi_c(t)=0$ has a unique solution $t_c$ in $[0,1]$, and $\lambda_{q,c}=\Lambda_{q,c}(t_c)$.
\end{itemize}

This finishes the proof of Lemma~\ref{lemmaFixedpoints} in the case of $q\in (0,1)$. 

Thus, only the case $q=1$ remains. By \eqref{Lambdadiff}, we see that $\Lambda_{1,c}'(1)=0$, and  $\Lambda_{1,c}'(t)$ has the same sign as $\phi_{c}(t)=\exp(-c(1-t)^k)-t$. For $t\in (0,1)$, let us let \[\Psi(t)=\frac{-\log(t)}{(1-t)^k},\]
we also set $\Psi(0)=+\infty=\lim_{t\to 0_+} \Psi(t)$.

It is straightforward to see that for $t\in [0,1)$, we have
\begin{enumerate}[\indent (a)]
\item We have $\phi_c(t)=0$, if and only if $c=\Psi(t)$.
\item We have $\phi_c(t)<0$, if and only if $c>\Psi(t)$.
\item We have $\phi_c(t)>0$, if and only if $c<\Psi(t)$.
\end{enumerate}
Also, $\phi_c(1)=0$ for all $c$.

We have
\[\Psi'(t)=\frac{t-1-kt\log(t)}{t(1-t)^{k+1}}.\]

Note that
\[\frac{d}{dt}\left(t-1-kt\log(t)\right)=1-k-k\log(t)\]
 is a monotonically decreasing function. 
Thus, $\Psi'(t)$ can have at most two zeros in $[0,1]$. Also note that $\lim_{t\to 0_+} \Psi(t)=\infty$ and $\lim_{t\to 1_-}\Psi(t)=\infty$. From these observations, we see that there is an $s\in (0,1)$ such that $\Psi$ is monotonically decreasing on $[0,s]$ and monotonically increasing on $[s,1)$. Note that $\Psi(s)>0$.

Thus,\begin{itemize}
    \item If $c<\Psi(s)$, then $\phi_c(t)=0$ has only one solution $t=1$. Moreover, $\Lambda_{1,c}$ is monotonically increasing in $[0,1]$, so $\lambda_{1,c}=\Lambda_{1,c}(1)$.
    \item If $c=\Psi(s)$, then $\phi_c(t)=0$ has two solutions $t=s$ and $t=1$. Moreover, $\Lambda_{1,c}$ is monotonically increasing in $[0,1]$, so $\lambda_{1,c}=\Lambda_{1,c}(1)$.

    \item If $c>\Psi(s)$, then $\phi_c(t)=0$ has three solutions $0<t_c<t_c'<t_c''=1$ in $[0,1]$. Moreover, $\Lambda_{1,c}$ has local maximums at $t_c$ and $t_c''=1$, and a local minimum at $t_c'$. Therefore, we have 
    $\lambda_{1,c}=\max (\Lambda_{1,c}(t_c),\Lambda_{1,c}(1))$.
    
\end{itemize}

Thus, Lemma~\ref{lemmaFixedpoints} follows.

\section{Convergence of persistence diagrams}

\subsection{Convergence of the verbose persistence diagrams}

In this section, we prove part (a) of Theorem~\ref{diagramconv}.
\begin{lemma}\label{CDF}
    Let $f:[0,\infty)^2\to [0,1]$ be a continuous function.

    \begin{enumerate}[(a)]
        \item Let $\eta_1,\eta_2,\dots$ be a tight sequence of probability measures on $\overline{\Delta}$ such that for all $r,s\ge 0$, we have
        \[\lim_{n\to\infty} \eta_n(([0,r]\times [0,s])\cap \overline{\Delta})=f(r,s).\]
        Then there is a unique probability measure $\eta$ on $\overline{\Delta}$ such that
        \[\eta(([0,r]\times [0,s])\cap \overline{\Delta})=f(r,s).\]
        \item Let $\eta_1',\eta_2',\dots$ be a sequence of probability measures on $\overline{\Delta}$ such that for all rational $r,s\ge 0$, we have
        \[\lim_{n\to\infty} \eta_n'(([0,r]\times [0,s])\cap \overline{\Delta})=f(r,s)=\eta(([0,r]\times [0,s])\cap \overline{\Delta}).\]
        Then $\eta_n'$ converges weakly to $\eta$.

        \item Let $(r_i,s_i)_{i=1}^\infty$ be an enumeration of the countable set $\{(r,s)\,:\,r,s\ge 0,\,r,s\in \mathbb{Q}\}$. Given a probability measure $\theta$ on $\overline{\Delta}$, we define
        \[\varrho(\theta)=\sum_{i=1}^{\infty} 2^{-i}\left|\theta(([0,r_i]\times [0,s_i])\cap \overline{\Delta})-f(r_i,s_i)\right|.\]

Let $\eta_1',\eta_2',\dots$ be a sequence of probability measures on $\overline{\Delta}$ such that $\lim_{n\to\infty} \varrho(\eta_n')=0$. 
        Then $\eta_n'$ converges weakly to $\eta$.
        \item Let $\mathcal{U}$ be an open subset of the probability measures on $\overline{\Delta}$ with respect to the weak topology such that $\eta\in\mathcal{U}$. Then, there is an $\varepsilon>0$ such that for all probability measures $\theta$ on $\overline{\Delta}$ satisfying $\varrho(\theta)<\varepsilon$, we have $\theta\in \mathcal{U}$.
    \end{enumerate}
\end{lemma}
\begin{proof} First, we prove part (a). 
Since $\eta_1,\eta_2,\dots$ is a tight sequence, by Prokhorov's theorem, we can find a sequence $0<n_1<n_2<\cdots$ such that as $i$ tends to infinity, $\eta_{n_i}$ converges weakly to some probability measure $\eta$ on $\overline{\Delta}$. Then
\[\eta(([0,r]\times [0,s])\cap \overline{\Delta})\ge\lim_{i\to\infty} \eta_{n_i}(([0,r]\times [0,s])\cap \overline{\Delta})=f(r,s).\]

Let $\varepsilon>0$. Since $f$ is continuous, we can find  $r'>r$ and $s'>s$ such that $f(s',r')
\le f(s,r)+\varepsilon$.  

Then
\begin{align*}\eta(([0,r]\times [0,s])\cap \overline{\Delta})&\le \eta(([0,r')\times [0,s'))\cap \overline{\Delta})\\&\le \liminf_{i\to\infty} \eta_{n_i}(([0,r')\times [0,s'))\cap \overline{\Delta})\\&\le \liminf_{i\to\infty} \eta_{n_i}(([0,r']\times [0,s'])\cap \overline{\Delta}) \le f(r,s)+\varepsilon.\end{align*}

Tending to $0$ with $\varepsilon$, it follows that $\eta(([0,r]\times [0,s])\cap \overline{\Delta})=f(r,s)$.

This proves the existence part of (a). The uniqueness will follow from part (b), which we prove next. 

Let 
\begin{align*}
\mathcal{I}&=\{[0,r]\,:\,r\ge 0, r\in \mathbb{Q}\}\cup \{(r_1,r_2]\,:\,0\le r_1<r_2, \,r_1,r_2\in\mathbb{Q}\},\text{ and}\\
\mathcal{A}_P&=\{(I_1\times I_2)\cap \overline{\Delta}\,:\,I_1,I_2\in \mathcal{I}\}.
\end{align*}
It is straightforward to see that
\[\lim_{n\to\infty} \eta_n'(A)=\eta(A)\]
for all $A\in \mathcal{A}_P$. Thus, the statement follows from \cite[Theorem 2.3]{billingsley2013convergence}.

The last two statements clearly follow from the previous ones.
\end{proof}

 Let 
\[\overline{\Delta}(u)=\{(r,s)\in \overline{\Delta}\,:\,s>u\}.\]

\begin{lemma}\label{tightnesslemma}
    For any $h>0$, there is a $C$ such that for all $n$ and $u>0$, we have
    \[\mathbb{E}\xi_{\Ver,k-1}(\mathcal{Y}_n,\kappa_n)(\overline{\Delta}(u))\le C u^{-h}.\]
\end{lemma}
\begin{proof}
    Let $\overline{\mathcal{Y}_n}(t)$ be obtained from  ${\mathcal{Y}_n}(t)$ by adding all the subsets of $[n]$ of size at most $k$. That is, $\overline{\mathcal{Y}_n}(t)$ is just the usual Linial-Meshulam complex. By~\eqref{measurevsBetti0}, we have
\[\xi_{\Ver,k-1}(\mathcal{Y}_n,\kappa_n)(\overline{\Delta}(u))=1-\frac{\dim B_{k-1}({\mathcal{Y}_n}(u))}{{{n-1}\choose{k}}}=1-\frac{\dim B_{k-1}(\overline{\mathcal{Y}_n}(u))}{{{n-1}\choose{k}}}=\frac{\dim H_{k-1}(\overline{\mathcal{Y}_n}(u))}{{{n-1}\choose{k}}}.\]
Combining this with \cite[Example 3.7]{hino2019asymptotic}, the statement follows.
\end{proof}

By Lemma~\ref{tightnesslemma}, we see that the sequence \[\eta_n=\mathbb{E}\xi_{\Ver,k-1}(\mathcal{Y}_n,\kappa_n)\] is tight.  

Let 
\[f(r,s)=1-\exp(-r)-\hat{\beta}_{k-1}^{r,s}.\]

\begin{lemma}
The function $f(r,s)$ is continuous. 
\end{lemma}
\begin{proof}
Let $o<\varepsilon<1$ and $C<\infty$. Then a straightforward calculation shows that there is an $L=L_{\varepsilon,C}$ such that for all $(q,c,t)\in (\varepsilon,1]\times [0,C]\times [0,1]$, we have
\[\left|\frac{d}{dq} \Lambda_{q,c}(t)\right|\le L\quad\text{ and }\quad\left|\frac{d}{dc} \Lambda_{q,c}(t)\right|\le L.\]
Thus, for all  $(q,c)\in (\varepsilon,1]\times [0,C]$, $(q',c')\in (\varepsilon,1]\times [0,C]$ and $t\in [0,1]$, we have
\[\left|\Lambda_{q,c}(t)-\Lambda_{q',c'}(t)\right|\le L(|q-q'|+|c-c'|),\]
so
\[\left|\lambda_{q,c}-\lambda_{q',c'}\right|\le L(|q-q'|+|c-c'|).\]
The continuity of $f(r,s)$ follows easily.
\end{proof}

By \eqref{Bettilimit01}  and \eqref{measurevsBetti0}, for all $r,s\in [0,\infty)$, we have
\[\lim_{n\to\infty} \eta_n(([0,r]\times [0,s])\cap \overline{\Delta})=f(r,s).\]

Thus, by Lemma~\ref{CDF}, there is a unique probability measure $\hat{\xi}_{\Ver,k-1}$ on $\overline{\Delta}$ such that for all $r,s\in [0,\infty)$, we have
\[\hat{\xi}_{\Ver,k-1}(([0,r]\times [0,s])\cap \overline{\Delta})=f(r,s).\]

 It follows from \eqref{Bettilimit01}  and \eqref{measurevsBetti0} that for all $r,s\ge 0$, we have
 \[\lim_{n\to\infty}\xi_{\Ver,k-1}(\mathcal{Y}_n,\kappa_n)(([0,r]\times [0,s])\cap \overline{\Delta})=f(s,r)\quad\text{ in probability}.\]

Thus,
\[\lim_{n\to\infty}\varrho\left(\xi_{\Ver,k-1}(\mathcal{Y}_n,\kappa_n)\right)=0\quad\text{ in probability}.\]
   Combining this with  part (d) of Lemma~\ref{CDF}, part (a) of Theorem~\ref{diagramconv} follows.

\subsection{Convergence of the persistence diagrams}

In this section, we prove part (b) of Theorem~\ref{diagramconv}.

From part (a) of Theorem~\ref{diagramconv}, it follows that $\xi_{k-1}(\mathcal{Y}_n,\kappa_n)$ converges vaguely to $\hat{\xi}_{k-1}$ in probability. Thus, it is enough to prove that $\xi_{k-1}(\mathcal{Y}_n,\kappa_n)(\Delta)$ converges to $\hat{\xi}_{k-1}(\Delta)$ in probability.

We define
\[\Phi_{q,c}(t)=1-c(1-qt)^k-\frac{c}{q(k+1)}\left(1-(1-qt)^{k+1}-q(k+1)t(1-qt)^k\right).\]

Using the elementary estimate that $|\exp(-x)-(1-x)|\le x^2$ for all $x\ge 0$, we see that for all $q\in (0,1]$, $c\in [0,\infty)$ and $t\in [0,1]$, we have
\[|\Phi_{q,c}(t)-\Lambda_{q,c}(t)|\le c^2.\]

Since
\[\Phi_{q,c}'(t)=ckq(1-t)(1-qt)^{k-1},\]
we see that $\Phi_{q,c}'(t)\ge 0$ for all $t\in [0,1]$.

Thus,
\[\max_{t\in [0,1]} \Phi_{q,c}(t)=\Phi_{q,c}(1)=1-\frac{c}{q(k+1)}\left(1-(1-q)^{k+1}\right).\]

Therefore,
\begin{equation}\label{lambdaestimate}\left|1-\frac{c}{q(k+1)}\left(1-(1-q)^{k+1}\right)-\lambda_{q,c}\right|\le c^2.\end{equation}

Note that
\begin{align*}\hat{\xi}_{\Ver,k-1}(((x,x+c]\times (x,x+c])\cap\overline{\Delta})&=\hat{\beta}_{k-1}^{x+c,x}+\hat{\beta}_{k-1}^{x,x+c}-\hat{\beta}_{k-1}^{x,x}-\hat{\beta}_{k-1}^{x+c,x+c}\\&=\exp(-x)(1-\lambda_{\exp(-x),c}).\end{align*}

Combining this with~\eqref{lambdaestimate}, we obtain that 
\begin{equation}\label{eqlec2}\left|\hat{\xi}_{\Ver,k-1}(((x,x+c]\times (x,x+c])\cap\overline{\Delta})-\frac{c}{k+1}\left(1-(1-\exp(-x))^{k+1}\right)\right|\le c^2.\end{equation}



Given  $S\subset [0,\infty)$, let
\[D_S=\{(x,x)\,:\,x\in S\}.\]
\begin{lemma}\label{Dab}
    For any $0\le a<b<\infty$, we have
    \[\hat{\xi}_{\Ver,k-1}(D_{(a,b]})=\int_{a}^b \frac{1-(1-\exp(-x))^{k+1}}{k+1} dx.\]
\end{lemma}
\begin{proof}
    For a positive integer $\ell$, we define $\delta_\ell=2^{-\ell}(b-a)$, and
    \[W_\ell=\bigcup_{i=0}^{2^\ell-1} \left(a+i\delta_\ell\,,\,a+(i+1)\delta_\ell\right]^2\cap\overline{\Delta}.\]
    Note that $W_1\supset W_2\supset \cdots$, and $\cap_{\ell=1}^\infty W_\ell=D_{(a,b]}$. Thus,
    \begin{equation}\label{Dabp1}
        \lim_{\ell\to\infty} \hat{\xi}_{\Ver,k-1}(W_\ell)=\hat{\xi}_{\Ver,k-1}(D_{(a,b]}).
    \end{equation}
    By \eqref{eqlec2}, we have
    \begin{equation}\label{Dabp2}\left|\hat{\xi}_{\Ver,k-1}(W_\ell)-\delta_\ell\sum_{i=0}^{2^\ell-1} \frac{1-\left(1-\exp\left(-\left(a+i\delta_\ell\right)\right)\right)^{k+1}}{k+1} \right|\le (b-a)^22^{-\ell}.\end{equation}
    Observe that
    \begin{equation}\label{Dabp3}
        \lim_{\ell\to \infty}\delta_\ell\sum_{i=0}^{2^\ell-1} \frac{1-\left(1-\exp\left(-\left(a+i\delta_\ell\right)\right)\right)^{k+1}}{k+1} =\int_{a}^b \frac{1-(1-\exp(-x))^{k+1}}{k+1} dx,
        \end{equation}
    since on the left we have Riemann-sums of the integral on the right. 
    
    The statement follows by combining \eqref{Dabp1}, \eqref{Dabp2} and \eqref{Dabp3}.
\end{proof}

Combining Lemma~\ref{Dab} with the fact that $\hat{\xi}_{\Ver,k-1}(\{(0,0)\})=0$, we obtain 
\begin{equation}\label{totalDelta0}
    \hat{\xi}_{\Ver,k-1}(D_{[0,\infty)})=\int_{0}^\infty \frac{1-(1-\exp(-x))^{k+1}}{k+1} dx=\frac{1}{k+1}\int_0^1 \frac{1-(1-t)^{k+1}}{t}dt=\frac{1}{k+1}\sum_{j=1}^{k+1}\frac{1}j,\end{equation}
where in the last step, we used the fact that

\[\int_0^1 \frac{1-(1-t)^{k+1}}{t}dt=\sum_{j=1}^{k+1}\frac{1}j,\]
which can be proved by induction, once we observe that
\[\int_0^1 \frac{1-(1-t)^{k+1}}{t}dt-\int_0^1 \frac{1-(1-t)^{k}}{t}dt=\int_0^1 (1-t)^k dt=\frac{1}{k+1}.\]

Clearly, the formula given in \eqref{totalDelta} follows from \eqref{totalDelta0}.

For any $t\ge 0$, let 
\[\mathcal{Y}_n(t-)=\{\sigma\in \mathcal{Y}_n\,:\,\kappa_n(\sigma)<t\}.\]

We say that $\sigma\in {{[n]}\choose{k+1}}$ is promoting if there is a $\tau\subset \sigma$ such that $|\tau|=k$ and $\tau\notin \mathcal{Y}_n(\kappa_n(\sigma)-)$.

For simplicity of notation, let $\xi_n=\xi_{\Ver,k-1}(\mathcal{Y}_n,\kappa_n)$.
\begin{lemma}\label{prom1}
Almost surely
\[\xi_n(D_{[0,\infty)})=\frac{\left|\left\{\sigma\in {{[n]}\choose{k+1}}\,:\,\sigma\text{ is promoting}\right\}\right|}{{{n-1}\choose{k}}}.\]

\end{lemma}
\begin{proof}
It is clear from the construction of $\xi_n$ given in Section~\ref{secPD} that $\xi_n$ is purely atomic, and if we have an atom at $(s,s)$, then we must have a $\sigma \in {{[n]}\choose{k+1}}$ such that $\kappa_n(\sigma)=s$. By the definition of $\kappa_n$, almost surely $\kappa_n$ is injective.  

Thus, conditioned on the event that $\sigma$ is a unique element of  ${{[n]}\choose{k+1}}$ such that $\kappa_n(\sigma)=s$, let us try to understand the event that $\xi_n$ has an atom at $(s,s)$. We have
\begin{align*}\xi_n(\{(s,s)\})&=\xi_n(([0,s]\times [0,s])\cap \overline{\Delta}))-\sup_{r<s}\xi_n(([0,r]\times [0,s])\cap \overline{\Delta}))\\&=\frac{\dim B_{k-1}(\mathcal{Y}_n(s))-\dim Z_{k-1}(\mathcal{Y}_n(s-))\cap B_{k-1}(\mathcal{Y}_n(s))}{{{n-1}\choose{k}}}.\end{align*}

If $\sigma$ is not promoting, then 
\begin{multline*}\dim B_{k-1}(\mathcal{Y}_n(s))-\dim Z_{k-1}(\mathcal{Y}_n(s-))\cap B_{k-1}(\mathcal{Y}_n(s))\\=\dim B_{k-1}(\mathcal{Y}_n(s))-\dim Z_{k-1}(\mathcal{Y}_n(s))\cap B_{k-1}(\mathcal{Y}_n(s))=0.\end{multline*}

If $\sigma$ is promoting, then one can see that $\dim Z_{k-1}(\mathcal{Y}_n(s-))\cap B_{k-1}(\mathcal{Y}_n(s))=B_{k-1}(\mathcal{Y}_n(s-))$. Therefore, it follows that
\begin{multline*}\dim B_{k-1}(\mathcal{Y}_n(s))-\dim Z_{k-1}(\mathcal{Y}_n(s-))\cap B_{k-1}(\mathcal{Y}_n(s))\\=\dim B_{k-1}(\mathcal{Y}_n(s))-\dim B_{k-1}(\mathcal{Y}_n(s-))=1.\end{multline*}
Thus, the statement follows.
\end{proof}

\begin{lemma}\label{prom2}
\[\mathbb{E}\frac{\left|\left\{\sigma\in {{[n]}\choose{k+1}}\,:\,\sigma\text{ is promoting}\right\}\right|}{{{n-1}\choose{k}}}=\frac{1}{k+1}\int_0^n 1-\left(1-\left(1-\frac{x}n\right)^{n-k-1}\right)^{k+1}  dx.\]
\end{lemma}
\begin{proof}
    Let $\sigma\in {{[n]}\choose{k+1}}$, and $t\in [0,1]$. Condition on the event that $\kappa_n(\sigma)=tn$. Note that $\sigma$ is not promoting if and only if for all $k$ element subsets $\tau$ of $\sigma$, we have a $v\in [n]\setminus \sigma $ such that $\kappa_n(\tau\cup \{v\})<tn$. The probability of this event is
    \[(1-(1-t)^{n-k-1})^{k+1}.\]
    Integrating over all possible choices of $t$, we get that
    \[\mathbb{P}(\sigma\text{ is promoting})=\int_0^1 1-(1-(1-t)^{n-k-1})^{k+1} dt=\frac{1}{n}\int_0^n 1-\left(1-\left(1-\frac{x}n\right)^{n-k-1}\right)^{k+1}  dx .\]

    Summing this over all the possible ${{n}\choose {k+1}}$ choices of $\sigma$ the statement follows.
    \end{proof}

Therefore, on the one hand, combining Lemma~\ref{prom1}, Lemma~\ref{prom2} and \eqref{totalDelta0}, we have
\begin{align*}\liminf_{n\to\infty}\mathbb{E} \xi_{\Ver,k-1}(\mathcal{Y}_n,\kappa_n)(D_{[0,\infty)})&=\liminf_{n\to\infty} \frac{1}{k+1}\int_0^n 1-\left(1-\left(1-\frac{x}n\right)^{n-k-1}\right)^{k+1}  dx\\&\ge \frac{1}{k+1}\int_0^\infty \liminf_{n\to\infty} \left( 1-\left(1-\left(1-\frac{x}n\right)^{n-k-1}\right)^{k+1} \right) dx\\&=\frac{1}{k+1}\int_0^\infty 1-(1-\exp(-x))^{k+1} dx\\&= \hat{\xi}_{\Ver,k-1}(D_{[0,\infty)}). \end{align*}

On the other hand, from part (a) of Theorem~\ref{diagramconv}, for any $\varepsilon>0$, we have
\[\lim_{n\to\infty}\mathbb{P}\left( \xi_{\Ver,k-1}(\mathcal{Y}_n,\kappa_n)(D_{[0,\infty)})\le \hat{\xi}_{\Ver,k-1}(D_{[0,\infty)})+\varepsilon\right)=1.\]

Thus, $\xi_{\Ver,k-1}(\mathcal{Y}_n,\kappa_n)(\Delta)$ converge in probability to \[\hat{\xi}_{\Ver,k-1}(\Delta)= 1-\hat{\xi}_{\Ver,k-1}(D_{[0,\infty)})=1-\frac{1}{k+1}\sum_{j=1}^{k+1}\frac{1}j,\]
where the last equality follows from \eqref{totalDelta0}.

\subsection{Unbounded observables -- The proof of Theorem~\ref{thmunboundedobs}}

For simplicity of notation, let $\xi_n=\xi_{\Ver,k-1}(\mathcal{Y}_n,\kappa_n)$ and $\xi_\infty=\hat{\xi}_{\Ver,k-1}$.

Clearly, we may assume that $f$ is nonnegative. For any $L>0$, let $f_L(r,s)=\min(f(r,s),L)$. Then

\[\int f d \xi_n=\int f_Ld\xi_n+\int_{L}^\infty \xi_n(f^{-1}([x,\infty)))dx.\]

 Recall that earlier, we defined
\[\overline{\Delta}(u)=\{(r,s)\in \overline{\Delta}\,:\,s>u\}.\]

Choose $M<\infty$ such that $f(r,s)<M(1+s^\ell)$ for all $(r,s)\in \overline{\Delta}$. Assuming that $x>M$, we have

\[f^{-1}([x,\infty))\subset \overline{\Delta}\left(\sqrt[\ell]{\frac{x}{M}-1}\right).\]

Thus, assuming that $L>M$, we see that
\[\left|\int f d\xi_n-\int f_L d\xi_n\right|\le \int_{L}^\infty \xi_n\left(\overline{\Delta}\left(\sqrt[\ell]{\frac{x}{M}-1}\right)\right)dx. \]

Similarly,
\[\left|\int f d\xi_\infty-\int f_L d\xi_\infty\right|\le \int_{L}^\infty \xi_\infty\left(\overline{\Delta}\left(\sqrt[\ell]{\frac{x}{M}-1}\right)\right)dx. \]

By Lemma~\ref{tightnesslemma}, we see that
\[\mathbb{E}(\xi_n(\overline{\Delta}(t)))\le Ct^{-2\ell}.\]
It is straightforward to see that this implies that
\[\xi_\infty(\overline{\Delta}(t))\le Ct^{-2\ell}.\]

Thus, for all $n$ and $L>M$, we have
\[\mathbb{E}\left|\int f d\xi_n-\int f_L d\xi_n\right|\le C\int_{L}^\infty \left(\frac{x}{M}-1\right)^{-2} dx,\]
and
\[\left|\int f d\xi_\infty-\int f_L d\xi_\infty\right|\le C\int_{L}^\infty \left(\frac{x}{M}-1\right)^{-2} dx.\]

Since for any fixed $L$, the integral $\int f_L d\xi_\infty$ converges to $\int f_L d\xi_\infty$ in probability by Theorem~\ref{diagramconv}, Theorem~\ref{thmunboundedobs} follows easily by tending to infinity with $L$.

\section{The rank of sparse matrices -- The proof of Theorem~\ref{RankThm} }\label{SecRankThm}

Throughout this section, given a matrix $M_n$, we always assume that the columns of $M_n$ are indexed by $C_n$, the rows of $M_n$ are indexed by $R_n$ and the Tanner graph of $M_n$ is denoted by~$G_n.$ Similarly, given a matrix $M_n'$, the columns of $M_n'$ are indexed by $C_n'$, the rows of $M_n'$ are indexed by $R_n'$ and the Tanner graph of $M_n'$ is denoted by $G_n'$, and so on.

\subsection{A regularization result}

Let $G$ be a finite graph and let $U$ be a nonempty subset of $V(G)$, we define
\[\overline{\deg}_G(U)=\frac{1}{|U|}\sum_{u\in U} \deg_G(u).\]

Let $\mu,\nu$ be distributions on the nonnegative integers with finite second moments. As before, $\bar{\mu}$ and $\bar{\nu}$ denote the mean of $\mu$ and $\nu$, respectively. We assume that $\bar{\nu}>0$ and $\bar{\mu}>0$.

\begin{definition}
Let $M_n$ be a sequence of deterministic matrices over the reals. We say that $M_n$ is a $(\mu,\nu)$-regular sequence if
\begin{enumerate}[(i)]
    \item All the entries of $M_n$ are from the set $\{-1,0+1\}$;
    \item The set-rooted graph $(G_n,C_n)$ converges to a $\GW_*(\mu,\nu)$-tree;
    \item \[\lim_{n\to\infty}\overline{\deg}_{G_n}(C_n)=\bar{\mu};\]
    \item The set-rooted graph $(G_n,R_n)$ converges to a $\GW_*(\nu,\mu)$-tree;
    \item \[\lim_{n\to\infty}\overline{\deg}_{G_n}(R_n)=\bar{\nu}.\]
\end{enumerate}

\end{definition}

\begin{lemma}\label{LemmaReg}
Let $M_n$ be a sequence of deterministic matrices with entries from the set $\{-1,0,+1\}$. 
Assume that $(G_n,C_n)$ converges to a $\GW_*(\mu,\nu)$-tree. Then we can find a deterministic $(\mu,\nu)$-regular sequence of matrices $M_n'$ such that \[\lim_{n\to\infty}\left|\frac{\rank(M_n)}{|C_n|}-\frac{\rank(M_n')}{|C_n'|}\right|=0.\]

\end{lemma}

\subsection{The proof of the regularization result (Lemma~\ref{LemmaReg})}

\begin{lemma}\label{degreg}
    Let $(G_n,U_n)$ be a deterministic sequence of set rooted graphs. Assume that $(G_n,U_n)$ converge locally to $(G,o)$ and $\mathbb{E}\deg_G(o)<\infty$. Then there are $U_n'\subset U_n$ with the following properties. Let $G_n'$ be obtained from $G_n$ by deleting the vertices in $U_n\setminus U_n'$. Then
    \begin{enumerate}[(i)]
        \item $\lim_{n\to\infty} \frac{|U_n'|}{|U_n|}=1$;
        \item $(G_n',U_n')$ converges locally to $(G,o)$;
        \item 
        \[\lim_{n\to\infty}\overline{\deg}_{G_n'}(U_n')=\mathbb{E}\deg_G(o).\]
    \end{enumerate}
\end{lemma}
\begin{proof}
    Let $u_n$ be a uniform random element of $U_n$. Let us define
    \[d_n=\sup\{d\,:\,\mathbb{E}\mathbbm{1}(\deg_{G_n}(u_n)<d)\deg_{G_n}(u_n)<\mathbb{E}\deg_G(o)+d^{-1}\}.\]

    We claim that $d_n$ converges to infinity. Indeed, given any $d$, since $(G_n,U_n)$ converges to~$(G,o)$, we have 
    \[\lim_{n\to\infty} \mathbb{E}\mathbbm{1}(\deg_{G_n}(u_n)<d)\deg_{G_n}(u_n)=\mathbb{E}\mathbbm{1}(\deg_G(o)<d)\deg_G(o)\le \mathbb{E}\deg_G(o)<\mathbb{E}\deg_G(o)+d^{-1}. \]

    Thus, $d_n\ge d$ for any large enough $n$.

    Let
    \[U_n'=\{u\in U_n\,:\, \deg_{G_n}(u)<d_n\}.\]

    Since $d_n$ tends to infinity, for any $d$, provided that $n$ is large enough, we have
    \[\{u\in U_n\,:\, \deg_{G_n}(u)<d\}\subset U_n'.\]
    Thus, it follows that
    \[\liminf_{n\to\infty} \frac{|U_n'|}{|U_n|}\ge \mathbb{P}(\deg_G(o)<d).\]
    Tending to infinity with $d$, we obtain that
    \[\liminf_{n\to\infty} \frac{|U_n'|}{|U_n|}\ge \lim_{d\to\infty}\mathbb{P}(\deg_G(o)<d)=1.\]

    Fix a nonnegative integer $r$, and let $(H,p)$ be a deterministic rooted graph of depth at most~$r$. Let us choose a $d$ larger than the maximum degree of $H$. Let $U_{n,r,d}$ be the set of vertices $u\in U_n$ such that the $r$-neighborhood of $u$ contains a vertex of degree at least $d$. For any $\varepsilon>0$, assuming that $d$ is large enough, the probability that $B_r(G,o)$ contains a vertex of degree at least $d$ is less than $\varepsilon$. Then it follows that for all large enough $n$, we have $|U_{n,r,d}|<\varepsilon |U_n|$. Thus,
    \[\left|\mathbb{P}(u_n\in U_n'\text{ and }B_r(G_n',u_n)\cong(H,p))-\mathbb{P}(B_r(G_n,u_n)\cong(H,p))\right|\le \frac{|U_{n,r,d}|}{|U_n|}<\varepsilon.\]
    Also, for large enough $n$, we have
    \[\left|\mathbb{P}(B_r(G,o)\cong(H,p))-\mathbb{P}(B_r(G_n,u_n)\cong(H,p))\right|<\varepsilon,\]
    and since $\frac{|U_n'|}{|U_n|}$ tends to $1$, for all large enough $n$, we have
    \[\left|\mathbb{P}(u_n\in U_n'\text{ and }B_r(G_n',u_n)\cong(H,p))-\mathbb{P}(B_r(G_n',u_n)\cong(H,p)\, |\,u_n\in U_n')\right|<\varepsilon.\]
    Therefore, assuming that $n$ is large enough
    \[\left|\mathbb{P}(B_r(G,o)\cong(H,p))-\mathbb{P}(B_r(G_n',u_n)\cong(H,p)\, |\,u_n\in U_n')\right|<3\varepsilon.\]
    Thus, it follows that $(G_n',U_n')$ converges to $(G,o)$.

    Finally, on the one hand, we have
    \begin{multline*}\limsup_{n\to\infty} \overline{\deg}_{G_n'}(U_n')\le\limsup_{n\to\infty} \frac{|U_n|}{|U_n'|}\mathbb{E}\mathbbm{1}(\deg_{G_n}(u_n)<d_n)\deg_{G_n}(u_n)\\\le \limsup_{n\to\infty} \left(\mathbb{E}\deg_G(o)+d_n^{-1}\right)=\mathbb{E}\deg_G(o). \end{multline*}
    On the other hand, from Fatuo's lemma and the fact that $(G_n',U_n')$ converges to $(G,o)$, we have
    \[\liminf_{n\to\infty} \overline{\deg}_{G_n'}(U_n')\ge \mathbb{E}\deg_G(o).\]
    Thus,
    \[\lim_{n\to\infty} \overline{\deg}_{G_n'}(U_n')= \mathbb{E}\deg_G(o).\qedhere\]
\end{proof}

Let $f:\mathbb{Z}_{0\le} \to [0,\infty)$. Given a random rooted graph $(G,o)$ such that $0<\mathbb{E}f(\deg_G(o))<\infty$, the $f$-degree biased version of $(G,o)$ is a random rooted graph $(G^f,o^f)$ such that for every deterministic rooted graphs $(H,p)$ of depth at most $r\ge 1$, we have
\[\mathbb{P}(B_r(G^f,o^f)\cong (H,p))=\frac{f(\deg_H(p))}{\mathbb{E}f(\deg_G(o))} \mathbb{P}(B_r(G,o)\cong (H,p)).\]
Note that the law of $(G^f,o^f)$ is uniquely determined by these requirements. 

\begin{lemma}\label{bias1}
Let $(G_n,o_n)$ be a sequence of random rooted graphs with local weak limit $(G,o)$. Assume that $0<\mathbb{E}f(\deg_G(o))<\infty$ and $0<\mathbb{E}f(\deg_{G_n}(o_n))<\infty$ for all $n$.
\begin{enumerate}[(a)]
    \item If $\lim_{n\to\infty}\mathbb{E}f(\deg_{G_n}(o_n))=\mathbb{E}f(\deg_G(o))$, then $(G_n^f,o_n^f)$ converges locally to $(G^f,o^f)$.
    \item If $f$ is bounded, then $(G_n^f,o_n^f)$ converges locally to $(G^f,o^f)$.
\end{enumerate}
\end{lemma}
\begin{proof}
Part (a) follows straight from the definitions. Part (b) follows from part (a) and the fact that $\lim_{n\to\infty}\mathbb{E}f(\deg_{G_n}(o_n))=\mathbb{E}f(\deg_G(o))$ for a bounded $f$.
\end{proof}

When $f(i)=i$ for all, we simply call the $f$-degree biased version of $(G,o)$ the degree biased version of $(G,o)$, and we denote it by $(\hat{G},\hat{o})$. Let $\mathfrak{f}(i)=\mathbbm{1}(i>0)\frac{1}i$. 

\begin{lemma}\label{bias2}
Let $(G,o)$ be a random rooted graph such that $\deg_G(o)>0$ almost surely and $\mathbb{E}\deg_G(o)<\infty$, and let $(\hat{G},\hat{o})$ be its degree biased version. Then
\begin{equation}\label{Emff}\mathbb{E} \mathfrak{f}(\deg_{\hat{G}}(\hat{o}))=\frac{1}{\mathbb{E}\deg_G(o)}.\end{equation}
Moreover, the $\mathfrak{f}$-degree biased version of $(\hat{G},\hat{o})$ has the same distribution as $(G,o)$.
\end{lemma}
\begin{proof}
    We have
    \[\mathbb{E} \mathfrak{f}(\deg_{\hat{G}}(\hat{o}))=\sum_{i=1}^\infty \frac{\mathbb{P}(\deg_{\hat{G}}(\hat{o})=i)}{i}=\sum_{i=1}^\infty \frac{\mathbb{P}(\deg_{{G}}({o})=i)}{\mathbb{E}\deg_G(o)}=\frac{1}{\mathbb{E}\deg_G(o)}.\]

    To prove the second statement, let $(H,p)$ be a deterministic rooted graph of depth at most $r\ge 1$ such that $p$ has positive degree. Using \eqref{Emff} and the definitions, we see that
    \begin{align*}
        \mathbb{P}(B_r((\hat{G})^\mathfrak{f},(\hat{o})^\mathfrak{f})\cong (H,p))&=\frac{1}{\deg_H(p)\mathbb{E} \mathfrak{f}(\deg_{\hat{G}}(\hat{o}))}\mathbb{P}(B_r(\hat{G},\hat{o})=(H,p))\\&=\frac{\mathbb{E}\deg_G(o)}{\deg_H(p)}\mathbb{P}(B_r(\hat{G},\hat{o})=(H,p))\\&=\mathbb{P}(B_r({G},{o})=(H,p)).\end{align*}
  Thus, the second statement follows.  
\end{proof}

\begin{lemma}\label{bias3}
    Let $(G,o)$ and $(G_n,o_n)$  be random rooted graphs such that $\deg_G(o)>0$ and $\deg_{G_n}(o_n)>0$ almost surely, and their root has finite expected degree. Assume that $(\hat{G}_n,\hat{o}_n)$ converges locally to $(\hat{G},\hat{o})$. Then $(G_n,o_n)$ converges locally to $(G,o)$. Moreover,
    \[\lim_{n\to\infty} \mathbb{E}\deg_{G_n}(o_n)=\mathbb{E}\deg_{G}(o).\]
\end{lemma}
\begin{proof}
The first part follows from part (b) of Lemma~\ref{bias1} and Lemma~\ref{bias2}. To prove the second part, note that by \eqref{Emff}, we have
\[\lim_{n\to\infty} \mathbb{E}\deg_{G_n}(o_n)=\lim_{n\to\infty}\frac{1}{\mathbb{E}\mathfrak{f}(\deg_{\hat{G}_n}(\hat{o}_n))}=\frac{1}{\mathbb{E}\mathfrak{f}(\deg_{\hat{G}}(\hat{o}))}=\mathbb{E}\deg_G(o).\qedhere\]
\end{proof}

Let $\mu=(\mu_i)$ and $\nu=(\nu_i)$ be two distributions on the set of nonnegative integers with finite second moments such that $\bar{\mu}>0$ and $\bar{\nu}>0$.

Let $(G_1,o_1)$ be a $\GW_*(\mu,\nu)$-tree, and let $(G_2,o_2)$ be a $\GW_*(\nu,\mu)$-tree. Let $(\hat{G}_1,\hat{o}_1)$  and $(\hat{G}_2,\hat{o}_2)$ be their degree biased versions. 

\begin{lemma}\label{munuswitch}
Let $\hat{o}'$ be a uniform random neighbor of $\hat{o}_1$ in the graph $\hat{G}_1$. Then $(\hat{G}_1,\hat{o}')$ has the same distribution as $(\hat{G}_2,\hat{o}_2)$. 

\end{lemma}
\begin{proof}
Let $(H_1,p_1)$ be a $\GW(\mu',\nu')$-tree, and let 
$(H_2,p_2)$ be a $\GW(\nu',\mu')$-tree such that $(H_1,p_1)$ and $(H_2,p_2)$ are independent. Connect $p_1$ and $p_2$ with an edge to combine $H_1$ and $H_2$ into a single tree $H$. The statement follows by observing that the biroted tree $(\hat{G}_1,\hat{o}_1,\hat{o}')$ has the same distribution as $(H,p_1,p_2)$.
\end{proof}

\begin{lemma}\label{specreg}
Let $G_n$ be a sequence of finite bipartite graphs with color classes $C_n$ and $R_n$ such that $(G_n,C_n)$ converges to a $\GW_*(\mu,\nu)$-tree. Assume that $\nu_0=0$. Then we can choose $C_n'\subset C_n$ with the following properties. Let $G_n'$ be the graph obtained from $G_n$ by first deleting the vertices in $C_n\setminus C_n'$, then deleting all the vertices of $R_n$ which are isolated in this new graph. Let $R_n'=V(G_n')\setminus C_n'$. Then
\begin{enumerate}[(i)]
    \item $\lim_{n\to\infty}\frac{|C_n'|}{|C_n|}=1$;
    \item $(G_n',C_n')$ converges to a $\GW_*(\mu,\nu)$-tree;
    \item \[\lim_{n\to\infty}\overline{\deg}_{G_n'}(C_n')=\bar{\mu};\]
    \item $(G_n',R_n')$ converges to a $\GW_*(\nu,\mu)$-tree;
    \item \[\lim_{n\to\infty}\overline{\deg}_{G_n'}(R_n')=\bar{\nu}.\]
\end{enumerate}
\end{lemma}
 \begin{proof}
Using Lemma~\ref{degreg}, we can find $C_n'\subset C_n$ such that (i), (ii) and (iii) hold. We show that (iv) and (v) also hold. 

Let $c_n'$ be a uniform random element of $C_n'$. Let $(\hat{G}_n',\hat{c}_n')$ be the degree biased version of $(G_n',c_n')$. It follows from (ii), (iii) and Lemma~\ref{bias1} that $(\hat{G}_n',\hat{c}_n')$ converge weakly to the a degree biased $\GW_*(\mu,\nu)$-tree. Let $\hat{r}_n'$ be a uniform random neighbor of $\hat{c}_n'$. It follows from Lemma~\ref{munuswitch} that $(\hat{G}_n',\hat{r}_n')$ converges weakly to a degree biased $\GW_*(\nu,\mu)$-tree. Let $r_n$ be a uniform random element of $R_n'$. Let $(\hat{G}_n',\hat{r}_n)$ be the degree biased version of $(G',r_n)$. Then $(\hat{G}_n',\hat{r}_n)$ and $(\hat{G}_n,\hat{r}_n')$ have the same distribution. So $(\hat{G}_n,\hat{r}_n)$ converges weakly to a degree biased $\GW_*(\nu,\mu)$-tree. Thus, (iv) and (v) follow using Lemma~\ref{bias3}.
\end{proof}

Now we are ready to prove Lemma~\ref{LemmaReg}.

    We first prove the special case when $\nu_0=0$. By Lemma~\ref{specreg}, we can find  $C_n'\subset C_n$ and $R_n'\subset R_n$ such that for the submatrix $M_n'$ of $M_n$ determined by the rows in $R_n'$ and the columns in $C_n'$ is $(\mu,\nu)$-regular and 
    \[\lim_{n\to\infty}\frac{|C_n'|}{|C_n|}=1.\]

    Note that $|\rank(M_n)-\rank(M_n')|\le |C_n|-|C_n'|$. Thus,
    \begin{align*}\left|\frac{\rank(M_n)}{|C_n|}-\frac{\rank(M_n')}{|C_n'|}\right|&\le \left|\frac{\rank(M_n)-\rank(M_n')}{|C_n|}\right|+\frac{\rank(M_n')}{|C_n|}\left|1-\frac{|C|'}{|C|}\right|\\&\le \frac{|C_n|-|C_n'|}{|C_n|}+\left|1-\frac{|C_n|'}{|C_n|}\right|.\end{align*}
    Thus,
    \[\lim_{n\to\infty}\left|\frac{\rank(M_n)}{|C_n|}-\frac{\rank(M_n')}{|C_n'|}\right|=0.\]
    If $\nu_0>0$, then let $\nu^c=(\nu_i^c)$ be defined as
    \[\nu_i^c=\begin{cases} 0&\text{for }i=0,\\
    \frac{\nu_i}{1-\nu_0}&\text{for }i>0.\end{cases}\]
    Then apply the argument above with $\nu^c$ in place of $\nu$. It is easy to see that if we add an appropriate number of all zero rows to $M_n'$, then we can achieve all the required properties.

\subsection{A lower bound on the rank by spectral methods}

The goal of this section is to prove the following lemma.

\begin{lemma}\label{ranklemma1}
    Let $M_n$ be a deterministic $(\mu,\nu)$-regular sequence of matrices. Then
    \[\liminf_{n\to\infty} \frac{\rank(M_n)}{|C_n|} \ge 1-\max_{t\in [0,1]} \Lambda(t).\]
\end{lemma}

Given a locally finite graph $G$ on the vertex set $V$, let $A_G$ be the adjacency operator of $G$, which is densely-defined on the
subset of finitely supported functions of the Hilbert space $\ell^2(V)$ by the equation that for all $u,v\in V$, we have
\[\langle A_G e_v,e_u\rangle =\begin{cases}
1&\text{if $uv$ is an edge of $G$,}\\
0&\text{otherwise,}
\end{cases}\]
where $e_v$ is the characteristic vector of $v$.

Assuming that $A_G$ is essentially self-adjoint, one can define $f(A_G)$ for any measurable function $f:\mathbb{R}\to\mathbb{C}$. Moreover, by the Spectral theorem, given any $o\in V$, there is a unique probability
measure $\eta_{(G,o)}$ such that
\[\langle f(A_G)e_o,e_o \rangle=\int f(t) \eta_{(G,o)}(dt).\]

Let $(T,o)$ be a rooted tree. Given a vertex $v$ of $T$, let $D(v)$ be the set of children of $v$, and let $T_v$ be the subtree of $T$ induced by $v$ and all the descendants of $v$. 

Combining Proposition 6 and Proposition 7 of \cite{bordenave2011rank}, we obtain the following lemma. (Note that in the statement of \cite[Proposition 7]{bordenave2011rank}, it is assumed that $\mu=\nu$, but the argument of the proof can be easily extended to $\GW_*(\mu,\nu)$-trees.)
\begin{lemma}\label{lemmarecursion}Let $(T,o)$ be a $\GW_*(\mu,\nu)$-tree or a $\GW(\mu',\nu')$-tree. Almost surely
\begin{enumerate}[(i)]
    \item $A_T$ is essentially self-adjoint, and for all $v\in V(T)$, $A_{T_v}$ is essentially self-adjoint.
    \item With the notation \[x_v=\eta_{(T_v,v)}(\{0\}),\]
    we have
    \[x_o=\left(1+\sum_{i\in D(o)}\left(\sum_{j\in D(i)}x_j\right)^{-1}\right)^{-1}.\]
    Here we use the convention that $1/0=\infty$ and $1/\infty=0$.
\end{enumerate}

\end{lemma}

\begin{lemma}\label{infinitetreebound}
    Let $(T,o)$ be a $\GW_*(\mu,\nu)$-tree. Then
    \[\mathbb{E}\eta_{(T,o)}(\{0\})\le \max_{t\in[0,1]}\Lambda(t).\]
\end{lemma}
\begin{proof}
In this proof, we rely on the ideas of \cite[Lemma 10]{bordenave2011rank}.

Let $(T',o')$ be a $\GW(\mu',\nu')$-tree, and let
\begin{align*}X&=\eta_{(T',o')}(\{0\}),\\ t&=\mathbb{P}(X>0). \end{align*}
Let $X_1,X_2,\dots$ be independent copies of $X$, and let $K$ be a random variable with law $\nu'$ independent from $(X_i)$. Let $S=\sum_{i=1}^K X_i$, and let $S_0,S_1,S_2\dots,$ be independent copies of $S$. Let $N$ be a random variable with law $\mu'$ independent from $(S_i)$. Using Lemma~\ref{lemmarecursion} and the definition of $T'$, we see that $X$ has the same distribution as
\[\left(1+\sum_{i=1}^N S_i^{-1}\right)^{-1}.\]

Let $N_*$ be a random variable with law $\mu$ independent from $(S_i)$. Then $\eta_{(T,o)}(\{0\})$ has the same distribution as
\[\left(1+\sum_{i=1}^{N_*} S_i^{-1}\right)^{-1}.\]

Thus,
\begin{align}
\mathbb{E}\eta_{T,o}(\{0\})&=\mathbb{E} \left(1+\sum_{i=1}^{N_*} S_i^{-1}\right)^{-1}
\nonumber\\&=\mathbb{E}\left(1-\frac{\sum_{i=1}^{N_*} S_i^{-1}}{1+\sum_{i=1}^{N_*} S_i^{-1}}\right)\mathbbm{1}(S_i>0\text{ for all }1\le i\le N_*)\nonumber\\&=f(\mu,1-f(\nu',1-t))-\mathbb{E}\left(\sum_{j=1}^{N_*} \frac{ S_j^{-1}}{1+\sum_{i=1}^{N_*} S_i^{-1}}\right)\mathbbm{1}(S_i>0\text{ for all }1\le i\le N_*)\nonumber\\&=
f(\mu,1-f(\nu',1-t))-\bar{\mu}\mathbb{E} \frac{ S_0^{-1}}{1+S_0^{-1}+\sum_{i=1}^{N} S_i^{-1}}\mathbbm{1}(S_i>0\text{ for all }0\le i\le N).\label{eqsp1}
\end{align}

Note that
\begin{align} \frac{ S_0^{-1}}{1+S_0^{-1}+\sum_{i=1}^{N} S_i^{-1}}\mathbbm{1}(S_i>0\text{ for all }0\le i\le N)&= \frac{ \frac{1}{1+\sum_{i=1}^N S_i^{-1}}}{S_0+\frac{1}{1+\sum_{i=1}^{N} S_i^{-1}}}\mathbbm{1}(S_i>0\text{ for all }0\le i\le N)\nonumber\\&=\frac{X'}{S_0+X'}\mathbbm{1}(S_0>0),\label{eqsp2}
\end{align}
where $X'$ is a copy of $X$ independent from $S_0$.

Let $Y_0,Y_1,\dots,$ be independent copies of $X$ conditioned on the event that $X>0$. Let $I_1,I_2,\dots$ be Bernoulli$(t)$ random variables independent from $K$, and let $K'=\sum_{i=1}^K I_i$. Then $S_0$ has the same distribution as $\sum_{i=1}^{K'} Y_i$. So
\begin{equation}\label{eqsp3}\mathbb{E}\frac{X'}{S_0+X'}\mathbbm{1}(S_0>0)=t\mathbb{E}\frac{Y_0}{\sum_{i=0}^{K'} Y_i}\mathbbm{1}(K'>0).\end{equation}

Note that for a deterministic $k$,
\[\mathbb{E}\frac{Y_0}{\sum_{i=0}^{k} Y_i}=\frac{1}{k+1}.\]

Therefore,
\begin{align}t\mathbb{E}\frac{Y_0}{\sum_{i=0}^{K'} Y_i}\mathbbm{1}(K'>0)&=t\sum_{i=1}^\infty \nu_i'\sum_{k=1}^{i} {{i}\choose{k}}t^k(1-t)^{i-k}\frac{1}{k+1}\nonumber\\&=\sum_{i=1}^\infty \nu_i'\frac{1}{i+1}\sum_{k=1}^{i} {{i+1}\choose{k+1}}t^{k+1}(1-t)^{i-k}\nonumber\\&=\sum_{i=1}^\infty \nu_i'\frac{1}{i+1}(1-(1-t)^{i+1}-(i+1)t(1-t)^i)\nonumber\\&=\frac{1}{\bar{\nu}}\sum_{i=1}^\infty \nu_{i+1}-\frac{1}{\bar{\nu}}\sum_{i=1}^\infty \nu_{i+1} (1-t)^{i+1}-t\sum_{i=1}^\infty \nu_i' (1-t)^{i}\nonumber \\&=\frac{1}{\bar{\nu}}(1-f(\nu,1-t)-\bar{\nu}tf(\nu',1-t)).\label{eqsp4}\end{align}

Combining \eqref{eqsp1}, \eqref{eqsp2}, \eqref{eqsp3} and \eqref{eqsp4}, we obtain that
\[\mathbb{E}\eta_{T,o}(\{0\})=f(\mu,1-f(\nu',1-t))-\frac{\bar{\mu}}{\bar{\nu}}(1-f(\nu,1-t)-\bar{\nu}tf(\nu',1-t))=\Lambda(t)\le \max_{t'\in [0,1] }\Lambda(t').\qedhere\]
\end{proof}

Let us define the matrix
\[F_n=\begin{pmatrix}
0&M_n\\
M_n^T&0
\end{pmatrix}.\]
The rows and columns of the matrix $F_n$ are both indexed by $R_n\cup C_n$. Given $o_n\in C_n$, the spectral measure $\eta_{(F_n,o_n)}$ is defined as the unique probability measure on $\mathbb{R}$ such that for any measurable function $f:\mathbb{R}\to\mathbb{C}$, we have

\begin{equation}\label{Fnspectral}\langle f(F_n)e_{o_n},e_{o_n} \rangle=\int f(t) \eta_{(F_n,o_n)}(dt).\end{equation}

\begin{lemma}\label{lemmaspectral}
    For each $n$, let us choose a deterministic $o_n\in C_n$. Assume that there is a (deterministic) locally finite rooted tree $(T,o)$ and a sequence of non-negative integers~$r_n$ such that
    \begin{itemize}
        \item The adjacency operator $A_T$ is essentially self-adjoint;
        \item The sequence $r_n$ tends to infinity;
        \item For all $n$, $B_{r_n}(G_n,o_n)\cong B_{r_n}(T,o)$.
        
    \end{itemize}

    Then
    \[\limsup_{n\to\infty}\eta_{(F_n,o_n)}(\{0\})\le \eta_{(T,o)}(\{0\}).\]
\end{lemma}
\begin{proof}
Since $T$ is a tree, $B_{r_n}(G_n,o_n)$ is also a tree. Then it is easy to see that one can find a diagonal matrix $U_n$ such that all the diagonal entries of $U_n$ are either $1$ or $-1$, and if we consider the matrix $F_n'=U_n^*F_nU_n$, then given any two vertices $u$ and $v$ of $B_{r_n}(G_n,o_n)$, we have
\[F_n'(u,v)=\begin{cases}1&\text{if $uv$ is an edge of $G_n$,}\\0&\text{otherwise}.\end{cases}\]

Note that $U_n$ is unitary and $U_ne_{o_n}=\pm e_{o_n}$. Thus, for any measurable $f:\mathbb{R}\to \mathbb{C}$, we have
\[\langle f(F_n')e_{o_n},e_{o_n} \rangle=\langle U_n^*f(F_n)U_n e_{o_n},e_{o_n} \rangle=\langle f(F_n)U_ne_{o_n},U_ne_{o_n} \rangle=\langle f(F_n)e_{o_n},e_{o_n} \rangle.\]
Thus, it follows that $\eta_{(F_n,o_n)}=\eta_{(F_n',o_n)}$. In particular, $\eta_{(F_n,o_n)}(\{0\})=\eta_{(F_n',o_n)}(\{0\})$.

By \cite[Proposition 11]{bordenave2011rank}, we see that $\eta_{(F_n',o_n)}$ converges weakly to $\eta_{(T,o)}$. 

Thus,
\[\limsup_{n\to\infty}\eta_{(F_n,o_n)}(\{0\})=\limsup_{n\to\infty}\eta_{(F_n',o_n)}(\{0\})\le \eta_{(T,o)}(\{0\}).\qedhere\]
\end{proof}

Let $P_n:\mathbb{R}^{R_n\cup C_n}\to \mathbb{R}^{R_n\cup C_n}$ be the orthogonal projection to $\ker F_n$. Applying \eqref{Fnspectral} with the choice of $f(t)=\mathbbm{1}(t=0)$, we see the $\eta_{(F_n,o_n)}(\{0\})$ can be expressed as $\|P_n e_{o_n}\|_2^2$. 

Let $P_n':\mathbb{R}^{R_n}\to \mathbb{R}^{R_n}$ be the orthogonal projection to $\ker M_n^T$ and $P_n'':\mathbb{R}^{C_n}\to \mathbb{R}^{C_n}$ be the orthogonal projection to $\ker M_n$. Observing that $P_n=P_n'\oplus P_n''$, we see that for $o_n\in C_n$, we have
\[\|P_n'' e_{o_n}\|_2^2=\|P_n e_{o_n}\|_2^2=\eta_{(F_n,o_n)}(\{0\}).\]

Since the dimension of the image of an orthogonal  projection is given by the trace of the projection, we see that
\begin{equation}\label{dimkerproj}\dim \ker M_n=\sum_{o_n\in C_n} \|P_n'' e_{o_n}\|_2^2=\sum_{o_n\in C_n}\eta_{(F_n,o_n)}(\{0\}). \end{equation}

Now choose $o_n$ as a uniform random element of $C_n$. Let $(T,o)$ be a $\GW_*(\mu,\nu)$-tree. By Skorokhod's representation theorem, one can find a coupling of $o_1,o_2,\dots$ and $(T,o)$ such that almost surely there is a (random) sequence $r_1,r_2,\dots$ of nonnegative integers such that $\lim_{n\to\infty} r_n=\infty$ and for all $n$, $B_{r_n}(G_n,o_n)\cong B_{r_n}(T,o)$.

Then combining \eqref{dimkerproj} and Lemma~\ref{lemmaspectral}, we see that
\[\limsup_{n\to\infty} \frac{\dim\ker M_n}{|C_n|}=\limsup_{n\to\infty} \mathbb{E}\eta_{(F_n,o_n)}(\{0\})\le \mathbb{E}\limsup_{n\to\infty} \eta_{(F_n,o_n)}(\{0\})\le\mathbb{E} \eta_{(T,o)}(\{0\}).\]
Combining this with the upper bound provided by Lemma~\ref{infinitetreebound}, Lemma~\ref{ranklemma1} follows.

\subsection{An upper bound on the rank by leaf removal procedure}




The goal of this section is to prove the following lemma using the leaf removal algorithm which was first analyzed by Karp and Sipser~\cite{karp1981maximum}.

\begin{lemma}\label{ranklemma2}
Let $M_n$ be a deterministic $(\mu,\nu)$-regular sequence of matrices. Let $\alpha$ be the smallest solution of
 \[t=f(\mu',1-f(\nu',1-t))\]
 in the interval $[0,1]$.

Then
\[\limsup_{n\to\infty} \frac{\rank(M_n)}{|C_n|}\le 1-\Lambda(\alpha).\]
\end{lemma}

Given a rooted tree $(T,o)$ together with a proper two coloring with color classes $C$ and $R$, we define the sets $L_i=L_i(T,o,C)$ and $K_i=K_i(T,o,C)$ as follows. 

We set $L_0=K_0=\emptyset$, and for $i\ge 1$, we set
\begin{align*}
    L_i&=\{c\in C\,:\,\text{ all the children of $c$ are in }K_{i-1}\},\\
    K_i&=\{r\in R\,:\,\text{ at least one child of $r$ is in }L_i\}.
\end{align*}

Let $(T,o)$ be a $\GW(\mu',\nu')$-tree with a proper two coloring with color classes $C$ and $R$ such that $o\in C$. Let $(T',o')$ be a $GW(\nu',\mu')$-tree with a proper two coloring with color classes $C'$ and $R'$ such that $o'\in R'$. We define
\begin{align*}
\alpha_i&=\mathbb{P}(o\in L_i(T,o,C)),\\
\beta_i&=\mathbb{P}(o'\in K_i(T',o',C')).
\end{align*}

It is straightforward to see that these quantities satisfy the following equations:
\begin{align*}
\alpha_0&=0,\\
\beta_0&=0,\\
\alpha_i&=f(\mu',\beta_{i-1})&\text{ for all }i\ge 1,\\
\beta_i&=1-f(\nu',1-\alpha_i)&\text{ for all }i\ge 1.
\end{align*}

Thus, for all $i\ge 0$, we have
\[\alpha_{i+1}=f(\mu',1-f(\nu',1-\alpha_i)).\]
The well-know argument gives that
\[\alpha=\lim_{i\to \infty} \alpha_i\]
exists, and $\alpha$ is the smallest solution of
\[\alpha=f(\mu',1-f(\nu',1-\alpha))\]
in the interval $[0,1]$.

Now, given a bipartite graph with color classes $R$ and $C$, we define $\bar{L}_i=\bar{L}_i(G,C)$ and $\bar{K}_i=\bar{K}_i(G,C)$ as follows: We set $\bar{K}_0=\emptyset$, and for $i\ge 1$, we set
\begin{align*}
    \bar{L}_i&=\{c\in C\,:\,\text{ $c$ has at most one neighbor not contained in }\bar{K}_{i-1}\},\\
    \bar{K}_i&=\{r\in R\,:\,\text{ at least one the neighbors of $r$ is in }\bar{L}_i\}.
\end{align*}

The proof of the next lemma is straightforward.
\begin{lemma}\label{LiKilocal}\hfill
\begin{enumerate}[(a)]
    \item For $c\in C$, we can decide whether $c\in \bar{L}_i$ based only on the $2(i+1)$ neighborhood of $c$. 
    \item For $r\in R$, we can decide whether $r\in \bar{K}_i$ based only on the $2(i+1)$ neighborhood of $r$.
\end{enumerate}
\end{lemma}

Let $(T,o)$ be a $\GW_*(\mu,\nu)$-tree. Then $T$ is a bipartite graph. Let $R$ and $C$ be its color classes such that $o\in C$. Then
\begin{align*}
    \mathbb{P}(o\in \bar{L_i})&=\mathbb{P}(r\in K_{i-1}\text{ for all }r\in D(o))\\&\qquad\qquad+\sum_{r\in D(o)}\mathbb{P}(r\notin K_{i-1}\text{ and }r'\in K_{i-1}\text{ for all }r'\in D(o)\setminus\{r\})\\&=
    \sum_{k=0}^\infty \mu_k(\beta_{i-1}^{k}+k(1-\beta_{i-1})\beta_{i-1}^{k-1})\\&=
    f(\mu,1-f(\nu',1-\alpha_{i-1}))+(1-\beta_{i-1})\bar{\mu}f(\mu',\beta_{i-1})\\&= f(\mu,1-f(\nu',1-\alpha_{i-1}))+f(\nu',1-\alpha_{i-1})\bar{\mu}\alpha_i.
\end{align*}

Combining this with the assumption that $(G_n,C_n)$ converges locally to $(T,o)$ and Lemma~\ref{LiKilocal}, we obtain that
\begin{equation}\label{oinL}\lim_{n\to\infty}\frac{|\bar{L}_i(G_n,C_n)|}{|C_n|}=\mathbb{P}(o\in \bar{L}_i(T,C))=f(\mu,1-f(\nu',1-\alpha_{i-1}))+f(\nu',1-\alpha_{i-1})\bar{\mu}\alpha_i.\end{equation}

Let $(T,o)$ be a $\GW_*(\nu,\mu)$-tree. Then $T$ is a bipartite graph. Let $R$ and $C$ be its color classes such that $o\in R$. Then
\begin{align*}
    \mathbb{P}(o\in \bar{K_i})&=1-\mathbb{P}(c\notin L_i\text{ for all }c\in D(o))\\&=1-
    \sum_{k=0}^\infty \nu_k(1-\alpha_i)^k\\&=1-f(\nu,1-\alpha_i).
\end{align*}

Combining this with the assumption that $(G_n,R_n)$ converges locally to $(T,o)$ and Lemma~\ref{LiKilocal}, we obtain that
\begin{equation}\label{oinK}\lim_{n\to\infty}\frac{|\bar{K}_i(G_n,C_n)|}{|R_n|}=\mathbb{P}(o\in \bar{K}_i(T,C))=1-f(\nu,1-\alpha_i).\end{equation}

Let us define
\[C_n'=C_n\setminus\bar{L}_i(G_n)\text{ and }R_n'=R_n\setminus \bar{K}_i(G_n),\]
and let $M_n'$ be the submatrix of $M_n$ determined by the rows in $R_n'$ and the columns in $C_n'$. 

\begin{lemma}We have
$\rank(M_n)-\rank(M_n')=|\bar{K}_i(G_n)|$.
\end{lemma}
\begin{proof}
For $r\in \bar{K}_i(G_n)$, the level $\ell(r)$ of $r$ is defined as the unique $\ell$ such that $r\in \bar{K}_{\ell}(G_n)\setminus \bar{K}_{\ell-1}(G_n)$. Let $r_1,r_2,\dots,r_h$ be a list of the elements of $\bar{K}_{i}(G_n)$ such that $\ell(r_1)\le \ell(r_2)\le \dots$. Let $M_{n,0}=M_n$, and for $j>0$, let $M_{n,j}$ be obtained from $M_{n,j-1}$ by deleting the row index by $r_j$. Note that $M_{n,j-1}$ always has a column which has a unique non-zero entry in the row indexed by $r_j$. In particular, the row indexed by $r_j$ is linearly independent from all other rows of $M_{n,j-1}$. Thus, $\rank(M_{n,j})=\rank(M_{n,j-1})-1$. Therefore, $\rank(M_{n,h})=\rank(M_n)-|\bar{K}_i(G_n)|$. Observe that the columns of $M_{n,h}$ which are indexed by $C_n\setminus C_n'$ are all zero columns. Thus, $\rank(M_n')=\rank(M_{n,h})$, and the statement follows. 
\end{proof}

Combining this with the trivial estimate that  $\rank(M_n')\le |C_n'|=|C_n|-|\bar{L}_i(G_n)|$. We obtain that
\begin{equation}\label{rankestimate}\frac{\rank(M_n)}{|C_n|}\le 1-\frac{|\bar{L}_i(G_n)|-|\bar{K}_i(G_n)|}{|C_n|}. \end{equation}

\begin{lemma}\label{RnperCn}
We have 
\[\lim_{n\to\infty} \frac{|R_n|}{|C_n|}=\frac{\bar{\mu}}{\bar{\nu}}.\]
\end{lemma}
\begin{proof}
Expressing the number of nonzero entries of $M_n$ in two different ways, we obtain that
\[|C_n|\cdot\overline{\deg}_{G_n}(C_n)=|R_n|\cdot\overline{\deg}_{G_n}(R_n).\]
Thus, since $M_n$ is $(\mu,\nu)$-regular, we have
\[\lim_{n\to\infty} \frac{|R_n|}{|C_n|}=\lim_{n\to\infty}\frac{\overline{\deg}_{G_n}(C_n)}{\overline{\deg}_{G_n}(R_n)}=\frac{\bar{\mu}}{\bar{\nu}}.\qedhere\]
\end{proof}

Combining \eqref{rankestimate}, \eqref{oinL}, \eqref{oinK} and Lemma~\ref{RnperCn}, we see that
\begin{align}\limsup_{n\to\infty} &\frac{\rank(M_n)}{|C_n|}\nonumber\\&\le 1-\lim_{n\to\infty}\frac{|\bar{L}_i(G_n)|-|\bar{K}_i(G_n)|}{|C_n|}\nonumber\\&=1-\lim_{n\to\infty}\frac{|\bar{L}_i(G_n)|}{|C_n|}+\lim_{n\to\infty}\frac{|R_n|}{|C_n|}\cdot\frac{|\bar{K}_i(G_n)|}{|R_n|}\nonumber\\&=1-\left(f(\mu,1-f(\nu',1-\alpha_{i-1}))+f(\nu',1-\alpha_{i-1})\bar{\mu}\alpha_i-\frac{\bar{\mu}}{\bar{\nu}}(1-f(\nu,1-\alpha_i))\right).\label{ranki}\end{align}

Observing that
\begin{multline*}\lim_{i\to\infty}f(\mu,1-f(\nu',1-\alpha_{i-1}))+f(\nu',1-\alpha_{i-1})\bar{\mu}\alpha_i-\frac{\bar{\mu}}{\bar{\nu}}(1-f(\nu,1-\alpha_i))\\=f(\mu,1-f(\nu',1-\alpha))-\frac{\bar{\mu}}{\bar{\nu}}(1-f(\nu,1-\alpha)-\bar{\nu}\alpha f(\nu',1-\alpha))=\Lambda(\alpha),\end{multline*}
Lemma~\ref{ranklemma2} follows from~\eqref{ranki} by tending to infinity with $i$.

\subsection{Leaf removal procedure for the transpose matrix}

The goal of this section to prove the following lemma by applying Lemma~\ref{ranklemma2} to $M_n^T.$
\begin{lemma}\label{ranklemma3}
Let $M_n$ be a $(\mu,\nu)$-regular sequence of matrices. Let $\alpha'$ be the largest solution of
 \[t=f(\mu',1-f(\nu',1-t))\]
 in the interval $[0,1]$.

Then
\[\limsup_{n\to\infty} \frac{\rank(M_n)}{|C_n|}\le 1-\Lambda(\alpha').\]

\end{lemma}

\begin{claim}
    The map $t\mapsto f(\nu',1-t)$ is monotonically decreasing on $[0,1]$ and the restriction of this map gives a bijection between the solutions of $t=f(\mu',1-f(\nu',1-t))$ in the interval $[0,1]$, and the solutions of $u=f(\nu',1-f(\mu',1-u))$ in the interval $[0,1]$.
\end{claim}
\begin{proof}
The map $t\mapsto f(\nu',1-t)$ is clearly monotonically decreasing, so we focus on the second part of the statement. If $t\in [0,1]$ is a solution of 
\[t=f(\mu',1-f(\nu',1-t)),\]
then
\[f(\nu',1-t)=f(\nu',1-f(\mu',1-f(\nu',1-t))).\]
Thus, $u=f(\nu',1-t)$ is a solution of
\[u=f(\nu',1-f(\mu',1-u)).\]

If $u$ is a solution of $u=f(\nu',1-f(\mu',1-u))$, then the same argument as before gives that $t=f(\mu',1-u)$ is a solution of $t=f(\mu',1-f(\nu',1-t))$, moreover,
\[u=f(\nu',1-f(\mu',1-u))=f(\nu',1-t).\qedhere\]
\end{proof}

It follows from the claim that $\alpha^T=f(\nu',1-\alpha')$ is the smallest solution of \[\alpha^T=f(\nu',1-f(\mu',1-\alpha^T))\]
in the interval $[0,1]$. 

It is clear from the definitions that if $M_n$ is a $(\mu,\nu)$-regular sequence, then $M_n^T$ is a $(\nu,\mu)$-regular sequence. Thus, applying Lemma~\ref{ranklemma2} to $M_n^T$, we obtain that
\begin{align*}\limsup_{n\to\infty}& \frac{\rank(M_n^T)}{|R_n|}\\&\le 1-f(\nu,1-f(\mu',1-\alpha^T))+\frac{\bar{\nu}}{\bar{\mu}}\left(1-f(\mu,1-\alpha^T)-\bar{\mu}\alpha^T f(\mu',1-\alpha^T)\right)\\&= 1-f(\nu,1-f(\mu',1-f(\nu',1-\alpha')))\\&\qquad+\frac{\bar{\nu}}{\bar{\mu}}\left(1-f(\mu,1-f(\nu',1-\alpha'))-\bar{\mu}f(\nu',1-\alpha') f(\mu',1-f(\nu',1-\alpha'))\right)\\&= 1-f(\nu,1-\alpha')+\frac{\bar{\nu}}{\bar{\mu}}\left(1-f(\mu,1-f(\nu',1-\alpha'))-\bar{\mu}\alpha'f(\nu',1-\alpha')\right)\\&=\frac{\bar{\nu}}{\bar{\mu}}\left(1-\Lambda(\alpha')\right).
\end{align*}

Combining this with Lemma~\ref{RnperCn}, it follows that
\[\limsup_{n\to\infty} \frac{\rank(M_n^T)}{|C_n|}\le \limsup_{n\to\infty} \frac{|R_n|}{|C_n|}\cdot\frac{\rank(M_n^T)}{|R_n|}\le 1-\Lambda(\alpha').\]

Thus, we proved Lemma~\ref{ranklemma3}.

\subsection{Finishing the proof of Theorem~\ref{RankThm}}

\begin{lemma}\label{ranklemma4}
Let $M_n$ be a deterministic sequence of matrices such that $(G_n,C_n)$ converges to a $\GW_*(\mu,\nu)$-tree. Let $\alpha$ and $\alpha'$ be the smallest and largest solution of
 \[t=f(\mu',1-f(\nu',1-t))\]
 in the interval $[0,1]$, respectively. 

 Assume that $\max_{t\in [0,1]}\Lambda(t)=\max(\Lambda(\alpha),\Lambda(\alpha'))$.

Then
\[\lim_{n\to\infty} \frac{\rank(M_n)}{|C_n|}=1-\max_{t\in [0,1]}\Lambda(t).\]
\end{lemma}
\begin{proof}
By Lemma~\ref{LemmaReg}, we may assume that $M_n$ is $(\mu,\nu)$-regular. Assuming that $M_n$ is $(\mu,\nu)$-regular the statement follows by combining Lemma~\ref{ranklemma1}, Lemma~\ref{ranklemma2} and Lemma~\ref{ranklemma3}.
\end{proof}

Now we prove Theorem~\ref{RankThm}. Since $(G_n,C_n)$ converges locally to a $\GW_*(\mu,\nu)$-tree in probability,  we can find a coupling of $M_1,M_2,\dots$ such that $(G_n,C_n)$ converges locally to a $\GW_*(\mu,\nu)$-tree almost surely. Thus, Theorem~\ref{RankThm} follows from Lemma~\ref{ranklemma4}.

\section{Finding a nice basis of $Z_i(\mathcal{K})$ -- The proof of Lemma~\ref{lemmagoodbasis}}\label{seclemmagoodbasis}

We define a total preorder on the set $Z_i(\mathcal{K})$ as follows. For $c_1,c_2\in Z_i(\mathcal{K})$, we say that $c_1\le c_2$, if either
\begin{itemize}
    \item $b(c_1)<b(c_2)$, or
    \item $b(c_1)=b(c_2)$ and $d(c_1)\le d(c_2)$.
\end{itemize}

Note that $b(c)$ and $d(c)$ can take only finitely many values, since $\kappa$ can take only finitely many values. Thus, it follows easily that every nonempty subset of $Z_i(\mathcal{K})$ has a smallest and a largest element. (They might not be unique.)

We construct a basis $\mathcal{C}=\{c_1,c_2,\dots,c_{\dim Z_{i}(\mathcal{K})}\}$ of $Z_i(\mathcal{K})$ by choosing the basis vectors one by one as follows. Assume that we have already chosen $c_1,c_2,\dots,c_{j-1}$, then we choose $c_{j}$ as one of the smallest elements of $Z_i(\mathcal{K})\setminus \text{span}(c_1,\dots,c_{j-1})$.

\begin{lemma}\label{lemmamatroid}
Let $c\in Z_i(\mathcal{K})$, and write $c$ as $c=\sum_{j=1}^{\dim Z_i(\mathcal{K})} \gamma_j c_j$. Then for all $j$ such that $\gamma_j\neq 0$, we have $c_j\le c$. 
\end{lemma}
\begin{proof}
Take a $j$ such that $\gamma_j\neq 0$. Then $c$ and $c_j$ are both elements of $Z_i(\mathcal{K})\setminus \text{span}(c_1,\dots,c_{j-1})$. Thus, by the rule of choosing $c_j$, we must have $c_j\le c$.
\end{proof}

To prove Lemma~\ref{lemmagoodbasis}, it is enough to prove the following lemma. 
\begin{lemma}
Let $c=\sum_{j=1}^{\dim Z_i(\mathcal{K})} \gamma_j c_j$ be an element of $Z_i(\mathcal{K})$, and let
\[J=\{j\,:\,\gamma_j\neq 0\}.\]  Then for all $j\in J$, we have $b(c_j)\le b(c)$ and $d(c_j)\le d(c)$. 
\end{lemma}
\begin{proof}
    Assume that the statement is not true. Given $\mathcal{K}$ and $\kappa$, among the counter examples choose one where $b(c)$ as small as possible. 
    
    By Lemma~\ref{lemmamatroid}, we see that $J$ can be partitioned into the following two sets
    \begin{align*}J_1&=\{j\in J\,:\,b(c_j)<b(c)\},\\J_2&=\{j\in J\,:\,b(c_j)=b(c)\text{ and }d(c_j)\le d(c)\}.\end{align*}
    Note that for all $j\in J_2$, we have $b(c_j)\le b(c)$ and $d(c_j)\le d(c)$.

    We let $c'=\sum_{j\in J_1}\gamma_j c_j=c-\sum_{j\in J_2}\gamma_j c_j$. Note that $b(c')\le \max_{j\in J_1} b(c_j)<b(c)$. Thus, by the choice of $c$, we see that for all $j\in J_1$, we have $d(c_j)\le d(c')$. Since $c'=c-\sum_{j\in J_2}\gamma_j c_j$, we see that $d(c')\le \max(d(c),\max_{j\in J_2} d(c_j))=d(c)$. Therefore,
    for all $j\in J_1$, we have $b(c_j)\le b(c)$ and $d(c_j)\le d(c)$. Thus, it turns out that $c$ was not a counter example, which is a contradiction.
\end{proof}

\bibliography{references}
\bibliographystyle{plain}

\bigskip

\noindent Andr\'as M\'esz\'aros, \\
{\tt meszaros@renyi.hu}\\
HUN-REN Alfr\'ed R\'enyi Institute of Mathematics,\\
Budapest, Hungary

\end{document}